\documentclass[11pt,reqno]{amsart}
\usepackage{amscd, amsmath, amssymb, amsthm, enumerate, multirow}
\usepackage[linktocpage=true]{hyperref}
\hypersetup{colorlinks,linkcolor=blue,urlcolor=blue,citecolor=blue}
\usepackage[table]{xcolor}
\usepackage{stmaryrd} 
\usepackage{bbm} 
\usepackage{leftidx} 
\usepackage{mathabx} 
\usepackage{rotating} 
\usepackage{wasysym} 
\usepackage[mmddyy]{datetime}
\usepackage{qtree}
\usepackage{float} 
\usepackage[all]{xy}
\usepackage{arydshln}
\makeatletter
\renewcommand*\env@matrix[1][*\c@MaxMatrixCols c]{%
  \hskip -\arraycolsep
  \let\@ifnextchar\new@ifnextchar
  \array{#1}}
\makeatother
\makeatletter
\newcommand{\labeltarget}[1]{\Hy@raisedlink{\hypertarget{#1}{}}}
\makeatother
\usepackage{todonotes}
\usepackage{pgf, tikz, colortbl, braids}
\usetikzlibrary{cd}
\tikzset{
	ch/.style={circle,draw,on chain,inner sep=2pt},
	chj/.style={ch,join},
	every path/.style={shorten >=4pt,shorten <=4pt}
	}

\numberwithin{equation}{subsection}
\counterwithin{table}{subsection}
\counterwithin{figure}{subsection}
\usepackage[margin = 0.9in]{geometry}
\def\<{\langle}
\def\>{\rangle}
\def\0{{\bar 0}}
\def\1{{\bar 1}}
\newcommand{\p}[1]{\ensuremath{\overline{#1}}}

\newcommand{\B}{\mrm{B}}

\newcommand{\cC}{\mathcal{C}}

\newcommand{\cF}{\mathcal F}
\newcommand{\cH}{\mathcal{H}}

\newcommand{\cO}{\mathcal{O}}

\newcommand{\CC}{\mathbb{C}}

\newcommand{\Dirr}{\Delta_{\mathrm{irr}}}

\newcommand{\Ext}{\mrm{Ext}}

\newcommand{\fa}{\ensuremath{\mathfrak{a}}}

\newcommand{\fb}{\ensuremath{\mathfrak{b}}}

\newcommand{\ff}{\ensuremath{\mathfrak{f}}}
\newcommand{\g}{\ensuremath{\mathfrak{g}}}
\newcommand{\fg}{\ensuremath{\mathfrak{g}}}

\newcommand{\fgl}{\ensuremath{\mathfrak{gl}}}
\newcommand{\fh}{\ensuremath{\mathfrak{h}}}
\newcommand{\fii}{\ensuremath{\mathfrak{i}}}
\newcommand{\fl}{\ensuremath{\mathfrak{l}}}
\newcommand{\fn}{\ensuremath{\mathfrak{n}}}
\newcommand{\fosp}{\ensuremath{\mathfrak{osp}}}
\newcommand{\fp}{\ensuremath{\mathfrak{p}}}
\newcommand{\fpsq}{\ensuremath{\mathfrak{psq}}}
\newcommand{\fq}{\ensuremath{\mathfrak{q}}}

\newcommand{\fsl}{\ensuremath{\mathfrak{sl}}}

\newcommand{\ft}{\ensuremath{\mathfrak{t}}}

\newcommand{\Hom}{\mrm{Hom}}

\newcommand{\Lie}{\operatorname{Lie}}

\newcommand{\mrm}{\mathrm}

\newcommand{\ninj}{\nabla_{\mathrm{inj}}}

\newcommand{\opH}{\operatorname{H}}
\newcommand{\opExt}{\operatorname{Ext}}

\newcommand{\RR}{\mathbb{R}}

\newcommand{\tif}{\textup{if }}

\newcommand{\ZZ}{\mathbb{Z}}

\renewcommand{\=}[1]{\overline{#1}}

\usepackage{enumerate}
\newenvironment{eq}{\begin{equation}}{\end{equation}}

\theoremstyle{definition}

\newtheorem{Def}{Definition}[subsection]
\newtheorem{definition}[Def]{Definition}
\newtheorem{defn}[Def]{Definition}

\newtheorem{ex}[Def]{Example}
\newtheorem{rmk}[Def]{Remark}

\theoremstyle{plain}
\newtheorem{prop}[Def]{Proposition}
\newtheorem{thm}[Def]{Theorem}
\newtheorem{mainthm}{Theorem}

\newtheorem{theorem}[Def]{Theorem}
\newtheorem{lemma}[Def]{Lemma}
\newtheorem{lem}[Def]{Lemma}
\newtheorem{cor}[Def]{Corollary}

\begin{document}
\title{Category $\mathcal{O}$ for Lie superalgebras}

\author{\sc Chun-Ju Lai}
\address{Institute of Mathematics\\ Academia Sinica \\Taipei\\ 106319, Taiwan}
\thanks{Research of the first author was supported in part by
the National Center of Theoretical Sciences}

\email{cjlai@gate.sinica.edu.tw}

\author{\sc Daniel K. Nakano}
\address
{Department of Mathematics\\ University of Georgia \\Athens\\ GA~30602, USA}
\thanks{Research of the second author was supported in part by
NSF grant DMS-2401184}
\email{nakano@uga.edu}

\author{\sc Arik Wilbert}
\address
{Department of Mathematics\\ University of South Alabama \\Mobile\\ AL~36688, USA}
\thanks{Research of the third author was supported by an AMS--Simons Grant for PUI Faculty}

\email{wilbert@southalabama.edu}

\maketitle

\begin{abstract} 
The authors define a Category $\mathcal{O}$ for any quasi-reductive Lie superalgebra $\mathfrak{g}$ with respect to a triangular decomposition. This much needed approach unifies many important constructions in the existing literature in a rigorous fashion. Our Category $\mathcal{O}$ encompasses all highest weight categories for Lie (super)algebras as well as specific examples which may not be highest weight categories. When the decomposition arises from a principal parabolic subalgebra $\mathfrak{p}$ of $\mathfrak{g}$, the Category $\mathcal{O}$ exhibits rich homological properties. For one, the authors show that in contrast to the case of a semisimple Lie algebra, the Category $\mathcal{O}$ is standardly stratified. 

Furthermore, the categorical cohomology of $\mathcal{O}$ is a finitely generated ring. This provides a first step towards developing a support variety theory for Category $\mathcal{O}$. It is shown that the complexity of modules in Category $\mathcal{O}$ is finite with an explicit upper bound given by the dimension of the subspace of the odd degree elements in $\mathfrak{g}$. This upgrades results known for $\mathfrak{gl}(m|n)$ to the more general setting. Our arguments are based on foundational connections between the categorical cohomology and the relative Lie superalgebra cohomology as well as the interplay between Category $\mathcal{O}$ for $\mathfrak{g}$ and the Category $\mathcal{O}$ for its corresponding Lie algebra $\mathfrak{g}_{\bar 0}$.

%
\end{abstract} 
\vskip 1cm

\section{Introduction} 
\subsection{} 

The BGG Category $\mathcal O$ introduced by Bernstein, I.\ Gelfand, and S.\ Gelfand in the early 1970's has been a central object of study in modern representation theory for nearly fifty years. The BGG Category $\mathcal O$ is associated to a complex semisimple Lie algebra $\mathfrak{g}_{\0}$ together with a choice of Borel subalgebra $\mathfrak{b}_{\0}\subseteq\mathfrak{g}_{\0}$. More precisely, it consists of certain possibly infinite-dimensional $U(\mathfrak{g}_{\0})$-modules subject to certain finiteness conditions. Category $\mathcal O$ has many remarkable structural properties. For example, it is a highest weight category as defined by Cline, Parshall and Scott \cite{CPS88}, and it admits a Kazhdan--Lusztig theory. Deep connections to geometry (flag varieties), topology (knot invariants and categorification), as well as combinatorics (via the Kazhdan--Lusztig conjectures) have been established 
(cf.  ~\cite{BB81, BK81, EW14, KL79,Str05, Su07}). 

Given a parabolic subalgebra $\mathfrak{p}_{\0}$ containing $\mathfrak{b}_{\0}$, one can define a parabolic Category $\mathcal O^{\mathfrak{p}_{\0}}$ generalizing Category $\mathcal O$. Many of the results for Category $\mathcal O$ have analogs in the parabolic setting, in particular Category  $\mathcal O^{\mathfrak{p}_{\0}}$  is also a highest weight category. Moreover, the representation type of the associated quasi-hereditary algebras was studied in~\cite{BN05}. 
For other axiomatic approaches via BGG algebras, see Irving \cite{I90}. 

For maximal parabolics in type A, the situation is particularly nice. In this case, these algebras can be constructed diagrammatically via so-called arc algebras (cf. ~\cite{BS11,CK14}). Much of the highest weight category structure, such as the Kazhdan--Lusztig polynomials, can be encoded combinatorially in the diagrams (cf. ~\cite{BS08}). Slight variations of the arc algebras can then be used to make connections between parabolic Category $\mathcal O^{\mathfrak{p}_{\0}}$ and perverse sheaves on Grassmannians~\cite{Str09}, Springer fibers~\cite{SW12}, knot homology~\cite{Kho02}, as well as Lie superalgebras~\cite{BS12}.

\subsection{}
Given the prominence of parabolic Category ${\mathcal O}$ for complex semisimple Lie algebras, it is natural to develop the theory of parabolic Category ${\mathcal O}$ for quasi-reductive Lie superalgebras ${\mathfrak g}={\mathfrak g}_{\0}\oplus {\mathfrak g}_{\1}$. Category ${\mathcal O}$ for Lie superalgebras is not a new concept. It has been developed on a case by case basis (cf.  \cite{Br03,Br04,CLW11, CLW12}), and introduced more generally in 
\cite{Maz14, AM15}.  

For this reason, it is our goal to develop a general theory that provides a unified picture encompassing different versions of Category $\mathcal O$ for Lie superalgebras that have appeared in the literature. 
When ${\mathfrak g}={\mathfrak g}_{\0}$, this includes the Category $\mathcal O$ for a complex semisimple Lie algebra. 
We remark that some of these categories, e.g. Brundan's Category $\cO_n$ \cite{Br04} for the isomeric Lie superalgebras, are not necessarily highest weight categories.
Blocks of certain parabolic Category $\cO$ which are highest weight categories with finitely many irreducibles have previously been studied in \cite{Maz14, CCC21}.  
One of our goals is to highlight the major differences between parabolic Category ${\mathcal O}$ for complex semisimple Lie algebras as opposed to 
parabolic Category ${\mathcal O}$ for quasi-reductive Lie superalgebras. For semisimple Lie algebras, the blocks are highest weight categories, and we 
demonstrate that the general behavior for quasi-reductive Lie superalgebras is that the blocks are standardly stratified categories with possibly infinitely many irreducible representations. 
In terms of the homological algebra, modules for parabolic Category ${\mathcal O}$ for semisimple Lie algebras have finite projective dimension, thus the cohomology lives in 
finitely many degrees, and have complexity equal to zero.  On the other hand, modules for Category ${\mathcal O}$ for quasi-reductive Lie superalgebras in general have infinite global dimension and 
have complexity which can be strictly positive.

\subsection{} 

In this paper, for a finite-dimensional quasi-reductive Lie superalgebra $\fg$ that admits a triangular decomposition 
${\mathfrak g}= \fn^+ \oplus \fa \oplus \fn^-$, where ${\mathfrak p}=\fa \oplus \fn^{+}$, we develop a general theory of Category $\cO^\fp$, which is a certain full subcategory of $U({\mathfrak g})$-modules.
Our Category $\cO^\fp$ encompasses numerous variations of Category $\cO$ in the literature. This includes the parabolic Category ${\mathcal O}$ for Lie algebras (by setting $\fa = \fl$ to be a Levi subalgebra),
and the category $\cF$ of finite-dimensional modules which are completely reducible over $\fg_\0$, by setting $\fa = \fg_\0$ as the even part of a Lie superalgebra of Type I.
Section~\ref{sec:examples} contains a compendium of examples covered by our theory.

One can then specialize our definition to the case where ${\fp}$ is a parabolic subalgebra as defined by 
Grantcharov and Yakimov \cite{GY13}. In fact, in several cases our main results will be stated in the case of $\fp$ being a {\em principal} parabolic subalgebra (see \cite[Section 1]{GY13}). 
With this setup, one can rely on using known results from the parabolic Category ${\mathcal O}^{{\mathfrak p}_{\0}}$ for ${\mathfrak g}_{\0}$ to prove fundamental theorems for ${\mathcal O}^{\mathfrak p}$. 
A precursor to our approach in which the parabolic Category ${\mathcal O}$ for Lie algebras was used to study Lie superalgebra cohomology appeared in Ph.D dissertation work of Maurer \cite{Ma19}. 

From a historical viewpoint, a similar setting via triangular decompositions for finite-dimensional modules was developed by Bellamy and Thiel \cite{BT18}. They generalized the ideas of Holmes and Nakano \cite{HN91} by considering the module category over a finite-dimensional algebra $A$ admitting a triangular decomposition. The authors were certainly inspired by the module constructions in \cite{HN91,BT18}. However, many of the arguments cannot be directly applied in our context because of the need to work with infinite-dimensional representations. A more elaborate setup using triangular decompositions can be found in work of Brundan and Stroppel \cite{BS24}.

\subsection{Main Results} 

The main results in the paper are  based on the choice of a principal parabolic subalgebra $\mathfrak{p}={\mathfrak l}\oplus {\mathfrak n}^{+}$ of $\mathfrak{g}$, where $\fa = \fl$ is a Levi subalgebra. The first main result establishes a 
direct connection between the cohomology in the Category ${\mathcal O}^{\mathfrak p}$ and the relative Lie superalgebra cohomology for the pair $({\mathfrak g},{\mathfrak l}_{\0})$. It is surprising to the authors that a formal general statement of this type has not appeared in the Lie superalgebra literature. 

\begin{mainthm} \label{T:cohomology}
If $\fp$ is a principal parabolic subalgebra of $\fg$, then 
the homological algebra for $\cO^\fp$ coincides with the relative cohomology for $(\fg, \fl_\0)$. 
For $M, N \in \cO^\fp$ and $i\geq 0$,
\[
\Ext^i_{\cO^\fp}(M,N) \cong \Ext^i_{(\fg, \fl_\0)}(M,N).
\]
\end{mainthm}

We utilize this important connection to relative Lie superalgebra cohomology to explicitly realize the cohomology ring and to establish finite generation results for Category ${\mathcal O}^{\mathfrak p}$. 
Statement (a) below indicates that the cohomology ring for $\cO^\fp$ has a tensor component that is the cohomology ring for the parabolic Category ${\mathcal O}$ for ${\mathfrak g}_{\0}$. 
The reader should also note that modules in $\cO^\fp$ can be infinite-dimensional. The statements in Theorem~\ref{T:finite generation} below are rather striking and non-trivial. 

\begin{mainthm} \label{T:finite generation}Let $\fp$ be a principal parabolic subalgebra of $\fg$. Then 
\begin{itemize} 
\item[(a)] The cohomology ring $R$ is a finitely generated ${\mathbb C}$-algebra. Furthermore, as a ring 
$$R\cong\Ext^{\bullet}_{{\cO^\fp}}(\mathbb{C},\mathbb{C})\cong S^{\bullet}({\mathfrak g}_{\1}^{*})^{G_{\0}}\otimes \opH^{\bullet}(\fg_{\0},{\mathfrak l}_{\0},{\mathbb C})\cong 
S^{\bullet}({\mathfrak g}_{\1}^{*})^{G_{\0}}\otimes \Ext^{\bullet}_{{\cO^{\fp_{\0}}}}(\mathbb{C},\mathbb{C}).$$
\item[(b)] If $M,\ N$ are finite-dimensional modules in $\cO^\fp$, then $\Ext^{\bullet}_{\cO^\fp}(M,N)$ is a finitely generated $R$-module. 
\item[(c)]  Assume further that ${\mathfrak g}$ is an almost simple Lie superalgebra (see Section \ref{sec:almostsimple}). Then 
for any $M$ finite-dimensional and $N$ in $\cO^\fp$, $\Ext^{\bullet}_{\cO^\fp}(M,N)$ is a finitely generated $R$-module. 
\end{itemize} 
\end{mainthm} 

The finite generation result in Theorem~\ref{T:finite generation} can be seen as a first step towards developing a support variety theory for our Category $\cO^\fp$. This would provide a geometric framework for an important piece of  the representation theory of $\cO^\fp$. It is well-known that the cohomology ring is in general too small to control the entire representation theory of $\cO^\fp$ (cf. \cite{BKN17}), and that a large geometric space, involving detecting subalgebras, will be needed as a key tool in studying the complexity of modules. 

In the original definition of the standardly stratified categories \cite[Definition~2.2.1]{CPS96},
there can only be finitely many simple modules,
which does not capture the general behavior of our Category $\cO^\fp$. Furthermore, the conditions involve using projective covers instead of injective hulls, and hence cannot be used to describe categories without enough projectives. 
This means that this definition of standardly stratified does not generalize the concept of a highest weight category. For instance, the category of rational $G$-modules where $G$ is a reductive group in positive characteristic 
is a highest weight category, but not standardly stratified. 

In Section \ref{sec:CFO}, we formulate a definition of a standardly stratified category that directly generalizes the original definition of a highest weight category due to Cline, Parshall, and Scott via injective modules (see Theorem \ref{HWCimpliesSS}). 
Moreover, our definition allows for the category to have infinitely many simples. Using Theorem~\ref{T:cohomology}, 
we are able to establish an orthogonality result between standard and costandard type modules. This enables us to prove that our Category $\mathcal O^{\mathfrak{p}}$ is standardly stratified with our new 
definition. 

\begin{mainthm}\label{T:stratification} If $\fp$ is a principal parabolic subalgebra of $\fg$, then 
the Category $\cO^\fp$ is standardly stratified as described in  Definition~\ref{def:ssc}.
\end{mainthm}

Finally, consider the {\em complexity} for $\cO^\fp$ given by, for $M \in \cO^\fp$, 
\eq\label{def:cM}
c_\fp(M) := r(\Ext^i_{\cO^\fp}(M, {\textstyle \bigoplus_{L \in \textup{Irr}\cO^\fp}} L )  ),
\endeq
where $r$ denotes the rate of growth.
As a consequence of Theorem~\ref{T:cohomology}, we are able to bound the complexity $c_\fp(M) = r(\Ext^i_{(\fg, \fl_\0)}(M, \bigoplus_{L \in \textup{Irr}\cO^\fp}L )  )$ using the rate of growth of the Koszul resolution.
\begin{mainthm}\label{T:complexity}
If $\fp$ is a principal parabolic subalgebra of $\fg$, then 
the complexity $c_\fp(M)$ of any $M \in \cO^\fp$ is bounded above by the dimension of the odd part $\fg_\1$.
\end{mainthm}

This result generalizes the result proved for ${\mathfrak g}=\mathfrak{gl}(m|n)$ due to Coulembier and Serganova,~\cite[Theorem 7.7]{CS17}.

\vskip .25cm
\noindent {\bf Acknowledgements.}
We acknowledge Jonathan Brundan, Chih-Whi Chen, Shun-Jen Cheng, and Volodymyr Mazorchuk for their input and insights about Category ${\mathcal O}$. Many of their comments were incorporated in this paper.
Moreover, the authors thank Academia Sinica, Shenzhen International Center for Mathematics, and the National Center for Theoretical Sciences (NCTS) for their hospitality and support during the completion of this paper.

\section{Categories with Filtered Objects}\label{sec:CFO}

\subsection{} In the representation theory of algebraic groups, Lie algebras, and finite groups, the modules need not be completely reducible. An effective method  to study the structure of various modules involves the  
use of filtrations. The notion of a  {\em highest weight category}, as introduced by Cline, Parshall, and Scott in \cite{CPS88}, has become a ubiquitous concept in providing a fundamental framework for studying filtrations whose subquotients are natural classes of representations. 

In this section, we present a definition of a {\em standardly stratified category} that differs from the one that was first given by Cline, Parshall, and Scott \cite{CPS96}. There are several reasons for stating this 
definition precisely. First, our definition of standardly stratified is a true generalization of highest weight categories as first defined by Cline, Parshall, and Scott~\cite{CPS88}. For example, the categories of rational $G$-modules where $G$ is a reductive algebraic group is a highest weight category. This category has infinitely many simple modules and does not have enough projectives. Thus, this category would not be considered a standardly stratified category under the definitions given in \cite{CPS96} and its variants in \cite{ADL98, Dl00, Fr07, LW15, BS24}. However, with our definition (see Definition~\ref{def:ssc}), the category of rational $G$-modules is indeed standardly stratified. 

Second, our definition will include blocks for Category ${\mathcal O}$ for Lie superalgebras where there are infinitely many simple modules. Moreover, our definition does not need the restricted assumptions 
with the poset $\Lambda$ having to be finite, and objects of $\cC$ being of finite length. The definition given in the following sections are much more compatible with our setting, and can be used in our prototypical examples  such as Brundan's Category $\cO_n$ \cite{Br04} for the queer Lie superalgebra $\fq(n)$.



\subsection{Highest Weight Categories}
For the convenience of the reader, we provide a precise definition of a highest weight category as first given in \cite[Definition~3.1]{CPS88}.
Recall that a locally Artinian category is a category $\cC$ such that any directed union of subobjects lies in $\cC$ and every object in $\cC$ is a union of subobjects of finite length.
The category $\cC$ is said to satisfy the Grothendieck condition if, for any subobject $B$ and any family $\{A_i\}_i$ of subobjects of an object $X$:
\eq \label{eq:GroCond}
B \cap ( \bigcup\nolimits_i A_i) = \bigcup\nolimits_i (B \cap A_i).
\endeq

\begin{definition}\label{def:hwc}
A {\em highest weight category} is a locally Artinian category $\cC$ with enough injectives, satisfying \eqref{eq:GroCond}, and 
 equipped with the following datum: 
\begin{enumerate}
\item[(a)]  An interval-finite poset $\Lambda$ (i.e., for all $x, y \in \Lambda, \#\{z\in \Lambda ~|~x< z < y\} < \infty$).
\item[(b)] A complete collection $\{L(\lambda)\}_{\lambda\in \Lambda}$ of non-isomorphic simple objects. 
\item[(c)]  A collection $\{\nabla(\lambda)\}_{\lambda\in \Lambda}$ of {\em costandard} objects in the following sense:
\begin{itemize}
\item[(i)] There is an embedding $L(\lambda) \hookrightarrow \nabla(\lambda)$.
\item[(ii)] If $[\nabla(\lambda)/L(\lambda): L(\mu)] \neq 0$, then $\mu < \lambda$.
\item[(iii)] For $\lambda,\mu\in \Lambda$, both $\dim \Hom_\cC(\nabla(\lambda), \nabla(\mu))$ and $[\nabla(\lambda):L(\mu)]$ are finite.
\end{itemize}
\item[(d)]  Each injective hull $I(\lambda)$ of $L(\lambda)$ has a $\nabla$-filtration $(F_i(\lambda))_{i\geq 0}$ such that
\begin{itemize}
\item[(i)] $I(\lambda) = \bigcup_{i\geq 0} F_i(\lambda)$.
\item[(ii)] $F_{i+1}(\lambda) / F_i(\lambda) \simeq \begin{cases}
\nabla(\lambda) &\tif i=0;
\\
\nabla(\mu_i) ~\textup{for some}~ \mu_i > \lambda &\tif i> 0 .
\end{cases}$
\item[(iii)] For any $\nu \in \Lambda$, $\#\{i\geq 0 ~|~ \mu_i = \nu\} < \infty$. 
\end{itemize}
\end{enumerate}
\end{definition}

\subsection{Standardly Stratified Categories}

For our purposes, it makes more sense to extend the definition of standardly stratified categories to include the original definition of the highest weight category. 
Recall that a quasi-poset (or, a proset) is a set $\Lambda$ having a preorder, i.e., a reflexive and transitive relation $\leq$.
Elements $i, j \in \Lambda$ are equivalent if and only if $i\leq j$ and $j \leq i$. 
Denote by $\overline{i}$ the equivalence class of $i$. 
Thus, $\overline{\Lambda} := \{ \overline{i} ~|~ i\in \Lambda\}$ is a poset whose partial order $\leq$ is given by $\overline{i} \leq \overline{j}$ if and only if $i \leq j$.

\definition\label{def:ssc}
 A {\em standardly stratified category} is a locally Artinian category $\cC$ with enough injectives, satisfying \eqref{eq:GroCond}, and 
 equipped with the following datum: 
\begin{itemize}
\item[(a)] An quasi-poset $\Lambda$ whose corresponding poset $\bar{\Lambda}$ is interval-finite.
\item[(b)] A complete collection $\{L(\lambda)\}_{\lambda\in \Lambda}$ of non-isomorphic simple objects. 
\item[(c)] A collection $\{\nabla(\lambda)\}_{\lambda\in \Lambda}$ of {\em costandard} objects such that
if $[\nabla(\lambda): L(\mu)] \neq 0$, then $\mu \leq \lambda$.
\item[(d)] Each injective hull $I(\lambda)$ of $L(\lambda)$ has a $\nabla$-filtration $(F_i(\lambda))_{i\geq 0}$ such that
\begin{enumerate}[(i)]
\item $I(\lambda) = \bigcup_{i\geq 0} F_i(\lambda)$.
\item $F_{i+1}(\lambda) / F_i(\lambda) \simeq \begin{cases}
\nabla(\lambda) &\tif i=0;
\\
\nabla(\mu_i) ~\textup{for some}~\mu_i \textup{ with } \bar{\mu}_i > \bar{\lambda} &\tif i> 0 .
\end{cases}$
\item For any $\nu \in \Lambda$, $\#\{i\geq 0 ~|~ \mu_i = \nu\} < \infty$. 
\end{enumerate}
\end{itemize}
\enddefinition

The costandard objects in a standardly stratified category have fewer requirements than those in a highest weight category. 
For example, if $[\nabla(\lambda)/L(\lambda): L(\mu)] \neq 0$ in a standardly stratified category, then $\mu$ can still be $ \lambda$.
Moreover, the finiteness conditions (c)(iii) in Definition~\ref{def:hwc} on Hom spaces between costandards and lengths of costandards are dropped.

\subsection{} Originally, Cline, Parshall, and Scott (CPS) severely relaxed the requirements of the highest weight category to introduce standardly stratified categories \cite[Definition~2.2.1]{CPS96}. 
The definitions therein considered only the case when $\Lambda$ is finite (i.e., there are only finitely many simple modules), and used projective covers instead of injective hulls. 
Moreover, our treatment differs from other approaches (see \cite{ADL98, Dl00, Fr07, LW15, BS24}), in that it is not assumed that one is working with a module category for an algebra. 
With the definition in the prior section, we can present a definitive statement. 

\begin{thm} \label{HWCimpliesSS} 
If $\cC$ is a highest weight category (as defined by CPS in Definition~\ref{def:hwc}) then $\cC$ is standardly stratified as defined in 
Definition~\ref{def:ssc}. 
\end{thm} 

\begin{ex} \label{ex:repG}
Let $G$ be a reductive algebraic group over an algebraically closed field of characteristic $p>0$. It is well-known the category of rational $G$-modules, 
$\cC=\text{Mod}(G)$, is a prototype for a highest weight category. The category $\cC$ has enough injective modules, but does not have enough projective modules. 
Furthermore, blocks can have infinitely many simple modules. 

From Theorem~\ref{HWCimpliesSS}, $\cC$ is standardly stratified, however, under older definitions as above, 
$\cC$ would not be viewed as standardly stratified. It is anticipated that Definition~\ref{def:ssc} will be useful in the study of the representation theory for quasi-reductive algebraic supergroups.  
\end{ex} 

\begin{ex} The prototype for a standardly stratified category ${\mathcal C}$ is the category of graded modules, $\cC=\cC_{gr}(A)$ for 
$A$ a finite-dimensional graded algebra over a field $k$ with a triangular decomposition $A\cong A^- \otimes A^0 \otimes A^+$ (as vector spaces). 

It was shown in \cite{HN91} and \cite[Section 4]{BT18}, that $\cC$ is standardly stratified under Definition~\ref{def:ssc}. Again, note that 
$\cC$ can have infinitely many simple modules. 
\end{ex}

\begin{rmk}
From a historical point of view, Cline, Parshall, and Scott (cf. \cite{CPS88}) developed highest weight categories to unify the theory of reductive algebraic group representations in positive characteristic and 
the Category ${\mathcal O}$ for semisimple Lie algebras over the complex numbers. The authors seek to use terminology that will reflect historical accuracy in the subject. 

Highest weight categories encompass the category $A$-fdMod of finite-dimensional modules over 
a quasi-hereditary finite-dimensional algebra $A$. Our definition of standardly stratified category naturally includes the finite-dimensional algebra setting explained in \cite[2.1, 2.2]{CPS96}. 
We anticipate that our categorical definition of standardly stratified should encompass ring theoretic framework in \cite{BS24}. 
The implications can be described by the diagram below:
\[
\begin{tikzcd}
{\begin{array}{c}
	A\textup{-fdMod}
	\\
	A:\textup{quasi-hereditary algebra}
\end{array}}
\ar[Rightarrow, r] 
\ar[Rightarrow, d]
	& 
\textup{Highest weight category}
\ar[Rightarrow, d]
\\
{\begin{array}{c}
	A\textup{-fdMod} 
	\\
	A:\textup{standardly stratified algebra}
\end{array}}
\ar[Rightarrow, r]
\ar[Rightarrow, d]
& 
\textup{Standardly stratified category}
\\
{\begin{array}{c}
	A\textup{-lfdMod} 
	\\
	A:\textup{essentially finite fibered}
	\\
	\textup{quasi-hereditary algebra}
\end{array}}
\ar[Rightarrow, ur, dashed, start anchor=east, end anchor=south]
& 
\end{tikzcd}
\]
Here, {\em essentially finite fibered quasi-hereditary algebras} are defined in \cite[Definitions 5.9, 5.18, Sections 2.2--2.4]{BS24}, 
and $A$-lfdMod is the category of locally finite-dimensional modules over an essentially finite-dimensional locally unital algebra $A$.
Examples of $A$-lfdMod that we know are standardly stratified.

It can be deduced from Example \ref{ex:repG} and from  \cite[Lemma 2.15]{BS24} that the ring theoretic formulations (on the left hand side) 
do not include the categorical formulations on the right hand side of the diagram. 
\end{rmk}


\section{Parabolic Category $\cO$ for Lie Superalgebras}
Let $\fg=\fg_{\0} \oplus \fg_{\1}$ be a finite-dimensional Lie superalgebra over $\CC$.
Often times we will refer to a subalgebra to mean a sub-superalgebra.
By a $\fg$-module we mean a $\ZZ_2$-graded module of its universal enveloping algebra $U(\fg)$.
Unless otherwise stated, by a morphism in any subcategory of the category $\fg$-Mod of $\fg$-modules we mean a graded even (i.e., parity preserving) morphism.
With this convention on the morphisms, the category of $\fg$-modules (with even morphisms) becomes an abelian category.
In this setup, one can then make use of the tools of homological algebra, along with the parity change functor $\Pi$ that interchanges the $\ZZ_2$-grading of a $\fg$-module.

\subsection{The Definition}
Let $\fa\subseteq \fg$ be a Lie sub-superalgebra.
Suppose that $\fg = \fa \oplus \fa'$ as vector spaces for some vector space $\fa'$ which is stable under $\fa$.
\defn\label{def:O}
Let $\cC_{(\fg, \fa)}$ be the full subcategory of $\fg$-Mod such that each object $M \in \cC_{(\fg, \fa)}$ is a direct sum of finite-dimensional $\fa$-modules.
\enddefn

Note that we are not making the assumption that our modules are finitely generated at this point. 
In fact, the (relative) injective modules that we will be working with need not be finitely generated. 


For the special case when $\fg = \fg_\0$ is a Lie algebra with $\fa = \ft_\0$ being its Cartan subalgebra, $C_{(\fg_0, \ft_0)}$ is the category of  weight modules in $\fg_\0$-Mod.

\defn\label{def:O}
Assume that $\fg$ admits a  decomposition $\fg = \fn^+ \oplus \fa \oplus \fn^-$ of Lie sub-superalgebras.
We define a new Category $\cO$ as follows. 
Let $\cO^\fp$ be the full subcategory of $\fg$-Mod such that each module $M \in \cO^\fp$ satisfies the following conditions:
\begin{enumerate}[(O1)]
\item $M$ is finitely generated over $U(\fg)$.
\item $M$ is a direct sum of finite-dimensional $\fa$-modules.
\item $M$ is locally finite with respect to $\fn^+$.
\end{enumerate}
\enddefn

We will see in Section \ref{sec:examples} that several interesting examples of $\cO^\fp$ arise from the notion of a principal parabolic subalgebra $\fp$ introduced by Grantcharov and Yakimov 
(see Section~\ref{sec:pps}). In the case of principal parabolic subalgebras, modules in the Category ${\mathcal O}$ will have weight space decompositions. 

\subsection{Principal Parabolic Subalgebras} \label{sec:pps}
Recall that a subalgebra of $\fg$ is parabolic if it contains a Borel subalgebra.
Following \cite{GY13}, we consider those parabolic subalgebras which are \footnote{The authors thank Mazorchuk for informing us that such parabolic subalgebras appeared first in the work \cite{DMP00} by Dimitrov, Mathieu, and Penkov.}{\em principal}. 
Let $\Phi$ be the set of roots of $\fg$ lying in a Euclidean space $V$.
\definition\label{def:pdcomp}
A subset $P \subseteq \Phi$ is called {\em parabolic} if the following criteria hold:
\begin{enumerate}[(\text{Case} 1)]
\item $\Phi = -\Phi$: $\Phi = P \cup (-P)$, and
$
\alpha, \beta \in P \textup{ implies either } \alpha + \beta \in P \textup{ or } \alpha + \beta \not\in \Phi.
$
\item $\Phi \neq -\Phi$: $\Phi \neq P$, and $P = \widetilde{P} \cap \Phi$ for some parabolic  set $\widetilde{P}$.
\end{enumerate}
Such a parabolic set $P$ affords a Levi decomposition $P = L \sqcup N^+$, where $L := P \cap (-P)$ and $N^+ := P \backslash (-P)$ in the former case when $\Phi = -\Phi$.
For the latter case, $L :=  \widetilde{L} \cap \Phi$ and $N^+ :=  \widetilde{N}^+ \cap \Phi$ for the Levi decomposition $\widetilde{P} = \widetilde{L} \sqcup \widetilde{N}^+$.
Moreover, there is a triangular decomposition $\fg = \fn^- \oplus \fl \oplus \fn^+$, where
\eq\label{def:lvg}
\fn^{+} := \bigoplus_{\alpha \in N^+} \fg_\alpha,
\qquad
\fn^{-} := \bigoplus_{-\alpha \in N^+} \fg_\alpha,
\qquad
\fl := \ft \oplus \Big(\bigoplus_{\alpha \in L} \fg_\alpha\Big).
\endeq
Here $\fg_\alpha$ is the corresponding root space for $\alpha$.
\enddefinition
\definition
For any $\cH \in V^*$, set
\eq
\Phi^-_\cH := \{ \alpha \in \Phi ~|~ \cH(\alpha) < 0\},
\quad
\Phi^0_\cH := \{ \alpha \in \Phi ~|~ \cH(\alpha) = 0\},
\quad
\Phi^+_\cH := \{ \alpha \in \Phi ~|~ \cH(\alpha) > 0\}.
\endeq 
A subset $P \subseteq \Phi$ is called {\em principal} if $P = \Phi^+_\cH \sqcup \Phi^0_\cH$ for some $\cH \in V^*$. Write $P = P_\cH$ in such a case.
\enddefinition
\begin{prop}\label{prop:Ldcomp}
If $P=P_\cH$ is a principal subset of $\Phi$, then $P$ is a parabolic subset.
Moreover,  $L = \Phi^0_\cH$ and $N^+ = \Phi^+_\cH$ gives  a  Levi decomposition of $P$ (which is not necessarily unique). 
\end{prop}
\proof
The proof for the case $\fg$ is simple is given in \cite[Proposition~1.6]{GY13}. In general, 
for almost simple $\fg$, the preceding argument still goes through.
\endproof
Combining Proposition~\ref{prop:Ldcomp} and \eqref{def:lvg}, any principal parabolic subset $P = P_\cH$ determines a triangular decomposition
\eq\label{eq:tridcomp}
\fg = \fn^-_\cH \sqcup \fl_\cH \sqcup \fn^+_\cH,
\quad 
\textup{where}
\quad
\fn^{\pm}_\cH := \bigoplus_{\alpha \in  \pm\Phi^+_\cH} \fg_\alpha,
\quad
\fl_\cH := \ft \oplus \bigoplus_{\alpha \in \Phi^0_\cH} \fg_\alpha.
\endeq

\subsection{Important facts about ${\mathcal O}^{{\mathfrak p}_{\0}}$} 
In this section, we collect some useful facts about the parabolic Category ${\mathcal O}^{{\mathfrak p}_{\0}}$ associated with a complex semisimple Lie algebra $\mathfrak{g}_{\0}$ and a parabolic subalgebra ${\mathfrak p}_{\0}\subseteq \mathfrak{g}_{\0}$, see~\cite{Hu08,Ku02}, which we will need in later sections. 

The category ${\mathcal O}^{{\mathfrak p}_{\0}}$ decomposes into direct summands ${\mathcal O}^{{\mathfrak p}_{\0}}_{\lambda}$ such that
$
{\mathcal O}^{{\mathfrak p}_{\0}}=\mathcal \bigoplus_{\lambda} {\mathcal O}^{{\mathfrak p}_{\0}}_{\lambda}
$,
where $\lambda$ runs over a complete set of representatives for the orbits under the dot action, 
and ${\mathcal O}^{{\mathfrak p}_{\0}}_{\mu}$ is the summand generated by the simple modules $L_\0(\lambda)$ with $\lambda\in W_{{\mathfrak p}_{\0}}\cdot \mu$. Let $Z_\0(\lambda)$ be the corresponding Verma module.

\begin{prop}\label{Cat O-Liealg prop}Let ${\mathcal O}^{{\mathfrak p}_{\0}}$ be the parabolic Category ${\mathcal O}$ associated with a complex semisimple Lie algebra $\mathfrak{g}_{\0}$. 
\begin{itemize} 
\item [(a)] Every $M\in {\mathcal O}^{{\mathfrak p}_{\0}}$ has finite length.
\item [(b)] The subcategory ${\mathcal O}^{{\mathfrak p}_{\0}}_{\mu}$ contains only finitely many simples, Vermas, and indecomposable projectives.
\item [(c)] The composition multiplicities $[Z_{\0}(\gamma): L_{\0}(\tau)]$ of simple modules in Verma modules are universally bounded in ${\mathcal O}^{{\mathfrak p}_{\0}}$.
\end{itemize} 
\end{prop}

\begin{proof} Part (a) and (b) are standard facts. For part (c), let $\gamma\in\Lambda_{\0}$. Then we have $[Z_{\0}(\gamma): L_{\0}(\tau)]\neq 0$ for only finitely many $\tau\in\Lambda_{\0}$ because every $\tau\in\Lambda_{\0}$ for which $L_{\0}(\tau)$ appears as a composition factor of $Z_{\0}(\gamma)$ must be in the finite $W_{{\mathfrak p}_{\0}}$-orbit of $\gamma$ under the dot-action. Since ${\mathcal O}^{{\mathfrak p}_{\0}}$ has finite length, we also have $[Z_{\0}(\gamma): L_{\0}(\tau)]<\infty$. Since there are finitely many blocks (up to Morita equivalence), we can take the maximum of these non-zero numbers to get a bound.
\end{proof} 

The cohomology for Category ${\mathcal O}^{{\mathfrak p}_{\0}}$ can be computed using the relative cohomology for the pair $(\mathfrak{g}_{\0},\mathfrak{l}_{\0})$ where 
$\mathfrak{l}_{\0}$ is the Levi subalgebra of $\mathfrak{p}_{\0}$. 

\begin{prop}
For $M,N\in{\mathcal O}^{{\mathfrak p}_{\0}}$, $\mathrm{Ext}^{\bullet}_{{\mathcal O}^{{\mathfrak p}_{\0}}}(M,N) \cong \mathrm{Ext}^{\bullet}_{(\mathfrak{g}_{\0},\mathfrak{l}_{\0})}(M,N)$. 
\end{prop}

Let $\ell(v)$ denote the length of an element $v\in W_{{\mathfrak p}_{\0}}$. Moreover, let $\leq$ be the Bruhat order on $W_{{\mathfrak p}_{\0}}$. The next result shows that 
there is parity vanishing for the Ext-groups between Verma and simple modules. 

\begin{prop}
For $v,w\in W_{{\mathfrak p}_{\0}}$, $v\leq w$, we have 
\[
P_{v,w}(t)=\sum\nolimits_{i\geq 0} t^i\dim \mathrm{Ext}^{\ell(w)-\ell(v)-2i}_{{\mathcal O}^{{\mathfrak p}_{\0}}}(Z_{\0}(v\cdot 0),L_{\0}(w\cdot 0)),
\]
where $P_{v,w}(t)$ is the Kazhdan--Lusztig polynomial associated with $v,w$.
\end{prop}

\section{Examples}\label{sec:examples}

A summary of important finite-dimensional Lie superalgebras can be found in Appendix \ref{S:QS}.
For these Lie superalgebras, the study of (parabolic) Category $\cO$ was initiated by Brundan (cf. \cite{Br03} for $\fgl(m|n)$ and \cite{Br04} for $\fq(n)$).
The parabolic Category $\cO$ for $\fgl(m|n)$ and $\fosp(m|n)$ are then used to establish the super duality by Cheng, Lam, and Wang in \cite{CLW12, CLW15}. 
Later on, parabolic Category $\cO$ for quasi-reductive Lie superalgebras has been studied by Mazorchuk, Chen, Cheng, Coulembier, and Wang \cite{Maz14, CW19, CW22, CCC21}.
The categories for Cartan series were studied by Duan, Shu, and Yao \cite{DSY24} based on work of Serganova \cite{Se05}.

\subsection{Category $\cO$ for Semisimple Lie Algebras} \label{sec:Oss}
Suppose that $\fg = \fg_\0$ is a semisimple Lie algebra.
Let $\fp = \fp_\0$ be a parabolic subalgebra with Levi subalgebra $\fa = \fa_\0$.
Then, our category  $\cO^\fp = \cO^{\fp_\0}$ is the usual parabolic Category $\cO$, which is a highest weight category (see \cite[Section 9]{Hu08}).

We remark that our axiom (O2) is a direct generalization of \cite[$(\cO^\fp2)$]{Hu08} in terms of direct sums of finite-dimensional modules over a subalgebra.
The existence of weight space decompositions is a consequence of (O2).

In this context, Theorem~\ref{T:cohomology} translates into the isomorphisms $\Ext_{\cO^\fp}^i(M, N) \cong \Ext^i_{(\fg, \fa)}(M,N)$ for all $i \geq 0$, and $M, N \in \cO^\fp$.
The isomorphisms for the special case where $\fa = \fh$ are well-known. 
It was stated in \cite{De80} (without proof),
and has appeared as an exercise in Kumar's book \cite[Exercise 3.3.E(1)]{Ku02}. 
A proof for this special case can be found in \cite[Theorem 9]{CM15}, in which the statement is equivalent to that the inclusion functor $\cO \hookrightarrow \mathcal{C}_{(\fg, \fa)}$ is {\em extension full}. 

Since $\Ext^i_{\cO^\fp}(M,N) = 0$ for sufficiently large $i$, Theorem~\ref{T:finite generation} holds trivially. 
Theorem~\ref{T:stratification} holds since a highest weight category is standardly stratified (with our definition). 
The inequality in Theorem~\ref{T:complexity} holds since the complexity is zero for any module in $\cO^\fp$.
\subsection{Lie Superalgebras of Type I} \label{sec:TypeI}
An important category of $\fg$-modules is the category $\cF$ of finite-dimensional $\fg$-modules which are completely reducible over $\fg_\0$. The category $\cF$  is in general not semisimple.
The category  $\cF$ is a highest weight category for each Lie superalgebra $\fg$ of Type I, 
and the proof is given separately, due to Brundan \cite{Br03} for $\fgl(m|n)$, to Cheng, Wang, and Zhang \cite{CWZ07} for $\fosp(2|2n)$, and to Chen \cite{C15} for $\fp(n)$.

In particular, a Type I Lie superalgebra $\fg$ admits a ${\mathbb Z}$-grading and the following decomposition:
\eq
\fg = \fn^+ \oplus \fg_\0 \oplus \fn^-,
\quad
\textup{where}
\quad
\fn^+ := {\mathfrak g}_{1}
\quad
\fn^- := {\mathfrak g}_{-1},
\endeq
Then, the category $\cF$ identifies with our Category $\cO^\fp$ by setting $\fa = \fg_\0$ and $\fp = \fg_\0 \oplus \fn^+$.  The standard (resp. costandard) modules are the Kac modules (resp. dual Kac modules). 

In this context, Theorem~\ref{T:stratification} is compatible with the fact that $\cF$ is a highest weight category,
while Theorems~\ref{T:cohomology}, \ref{T:finite generation} and \ref{T:complexity} are proved by Boe, Kujawa and Nakano in \cite[Corollary 2.4.2, Theorem 2.5.1]{BKN10a}, \cite[Theorems 2.5.2--3]{BKN10a}, and in  \cite[Theorem 2.5.1]{BKN10b}.

\subsection{Quasi-reductive Lie Superalgebras} 
Suppose that $\fg$ is a quasi-reductive (not necessarily simple) Lie superalgebra.
Mazorchuk's parabolic Category $\cO$ \cite{Maz14}  is an example of our Category $\cO^\fp$, where $\fp$ is a {principal parabolic subalgebra} as defined by Grantcharov and Yakimov \cite{GY13}.
In particular, such a Category $\cO^\fp$ recovers  Brundan's Category $\cO_{m|n}$  \cite{Br03} for $\fgl(m|n)$ and Category $\cO_n$ \cite{Br04} for $\fq(n)$, as well as the parabolic categories over $\fgl(m|n)$ and over $\fosp(m|n)$ which appeared in the super duality \cite{CLW15} established by Cheng, Lam and Wang.

While the irreducible characters associated with these categories are known, their extension structures are less understood, compared to the categories mentioned in Sections \ref{sec:Oss}--\ref{sec:TypeI}. 
To our best knowledge, the categorical extensions of $\cO^\fp$ has not been identified with relative cohomology,
and we believe that our Theorems \ref{T:cohomology}, \ref{T:finite generation} and \ref{T:complexity} are new in this context.

For Theorem~\ref{T:stratification}, only a special case in which $\cO^\fp$ is a highest weight category is known.
To be precise, a sub-collection of Mazorchuk's parabolic categories is later studied by Chen, Cheng and Coulembier, 
under a technical assumption that $\textup{Par}(\fg, \fp_\0) \neq \varnothing$, where $\textup{Par}(\fg, \fp_\0)$ is the set of reduced parabolic subalgebras $\fa \subseteq \fg$ such that $\fa_\0 = \fp_\0$  (see  \cite[Sections 1.4 and 3.1]{CCC21}).
In particular, their assumption excludes $\fq(n)$ from their discussion, 
since the assumption implies that the principal parabolic subalgebra $\fp$ must have at most one purely even Levi subalgebra, and hence the Cartan subalgebra must be purely even.
However, the Cartan subalgebra of $\fq(n)$ is not purely even.
Under this assumption, a highest weight category result \cite[Theorem 3.1(iii)]{ CCC21} is obtained.

In general, we expect that in most cases, Category $\cO^\fp$ for Lie superalgebras will not be a highest weight category.
\subsection{Lie Superalgebras affording Bott-Borel-Weil Parabolics}
Recall the detecting subalgebras $\ff$ constructed in \cite[Section 8.9]{BKN10a}.
Suppose that $\fg$ is a Lie superalgebra that affords Bott-Borel-Weil (BBW) parabolic subalgebras \cite[Definition 4.11.1]{GGNW21} introduced by D. Grantcharov, N. Grantcharov, Nakano, and Wu.
Let $G$ be the corresponding supergroup scheme with $\mathop{\textup{Lie}}G = \fg$, and $P$ be the parabolic subgroup corresponding to a BBW parabolic subalgebra $\fp \subseteq \fg$.
Note that  $\fp$ is automatically a principal parabolic subalgebra of $\fg$, and is used to obtain a Bott-Borel-Weil result in the sense that the Poincar\'e polynomial 
\eq
p_{G,P}(t) := \sum\nolimits_{j=0}^\infty t^j \dim R^j\textup{ind}_P^G(\CC)
\endeq
 identifies with a Poincar\'e polynomial for a certain finite reflection group, specialized to a power of $t$. 
This result particularly gives information about the cohomology of the trivial line bundle $\mathcal{H}^{\bullet}(G/P, \mathcal{L}(0)) = R^{\bullet}\textup{ind}_P^G(\CC)$.

In this context, there is a triangular decomposition for $\fg$, given by
setting $\fa = \ff$ and setting $\fp$ to be the corresponding BBW parabolic subalgebra. 

In particular, for $\fq(n)$, the detecting subalgebra $\ff \subseteq \fq(n)$, by construction, can be described as 
\eq
\ff_\0 = \{  \left(\begin{smallmatrix} A&0\\0&A\end{smallmatrix}\right) ~|~ A \in \fgl_n \textup{ is diagonal}\},
\quad
\ff_\1 = \{  \left(\begin{smallmatrix} 0&B\\B&0\end{smallmatrix}\right) ~|~ B \in \fgl_n \textup{ is diagonal}\}.
\endeq
Moreover, the BBW parabolic subalgebra $\fp$ is constructed by explicitly choosing a separating hyperplane $\cH$ such that $\Phi^0_\cH = \varnothing$, and that $\fp = \ff \oplus {\mathfrak n}^{+}$ agrees with the positive Borel subalgebra $\fb \subseteq \fq(n)$.
Hence, Brundan's Category $\cO_n$ can be identified with our Category $\cO^\fp$ in  this setup.
 
We anticipate that, among all categories $\cO^\fp$ with respect to an arbitrary principal parabolic subalgebra of a quasi-reductive Lie superalgebra, these categories $\cO^\fp$ associated with the BBW parabolic subalgebras  are the 
ones that most reflect the general superalgebra theory, in light of the good  cohomological properties provided in \cite{GGNW21}.
%
%
%
%
%
%

\section{Relative Lie Superalgebra Cohomology}  

\subsection{Relative cohomology for rings} 

For rings there is a notion of relative cohomology that can be extended to the situation for superalgebras. Let $R$ be a superalgebra over a field $k$ and $S$ a subsuperalgebra.  
One needs the notion of $(R,S)$-exact sequences and $(R,S)$-projectivity. 

\begin{definition} Let $R$ be a superalgebra and $S$ be a subalgebra of $R$. 
\begin{itemize} 
\item [(a)] Consider the sequence of $R$-supermodules and even $R$-supermodule 
homomorphisms:  
\[
\dotsb \to M_{i-1} \xrightarrow{f_{i-1}} M_{i} \xrightarrow{f_{i}}  M_{i+1} \to \dotsb 
\]
We say this sequence is \emph{$(R,S)$-exact} if it is exact as a 
sequence of $R$-supermodules and if, when viewed as a sequence of $S$-supermodules,  
$\text{Ker } f_{i}$ is a direct summand of $M_{i}$ for all $i.$  
\item[(b)] An $R$-supermodule $P$ is \emph{$(R,S)$-projective} if given any $(R,S)$-exact sequence 
\[
0 \to M_{1} \xrightarrow{f} M_{2} \xrightarrow{g}  M_{3} \to 0,
\] 
and $R$-supermodule homomorphism $h: P \to M_{3}$ there is an $R$-supermodule map $\tilde{h}:P \to M_{2}$ satisfying 
$g \circ \tilde{h}=h.$  
\item[(c)] An $(R,S)$-projective resolution of an $R$-supermodule $M$ is an $(R,S)$-exact sequence 
\[
\dotsb \xrightarrow{\delta_{2}} P_{1} \xrightarrow{\delta_{1}} P_{0} \xrightarrow{\delta_{0}} M \to 0,
\]
where each $P_{i}$ is an $(R,S)$-projective supermodule.
\end{itemize} 
\end{definition} 

An equivalent definition of an $R$-module, $M$, being $(R,S)$-projective is that $M$ is a direct summand of 
$R\otimes_{S} N$ for some $S$-module $N$. From this it is clear that any projective $R$-module is $(R,S)$-projective. 
Furthermore, any $R$-supermodule $M$ admits an $(R,S)$-projective resolution. We can define cohomology as follows: let 
$P_{\bullet} \to M$ be a $(R,S)$-projective resolution and let 
\eq
\Ext^{i}_{(R,S)}(M,N)={\opH}^{i}(\Hom_{R}(P_{\bullet}, N)).
\endeq
It is also natural to define the notion of $(R,S)$-injective modules and $(R,S)$-injective resolutions, and one can define 
$\text{Ext}_{(R,S)}^{\bullet}(M,N)$ using an injective $(R,S)$-injective resolution for $N$. 

In the many cases where $R=U(\fg)$ and $S=U(\fh)$ and $\fh$ is a Lie subsuperalgebra of $\fg$, the 
$(R,S)$-cohomology has a nice concrete interpretation via relative Lie superalgebra cohomology. Proofs can be found in \cite[Section 3.1]{Ku02} that can be generalized to the super setting. 

\begin{theorem}\label{T:relatescohom2} Let $\fg$ be a Lie superalgebra, $\fh$ a Lie subsuperalgebra, and 
$M,N$ $\fg$-supermodules.  Assume that either (i) $\fg$ is finitely semisimple as a $\fh$-supermodule under 
the adjoint action or (ii) $\fg=\fh \oplus {\mathfrak q}$ where ${\mathfrak q}$ is a Lie subsuperalgebra of $\fg$. Then
\begin{equation*}
\Ext^{\bullet}_{({U}(\fg),{U}(\fh))}(M, N) \cong {\opH}^{\bullet}(\fg, \fh, M^{*} \otimes N).
\end{equation*}
\end{theorem}
From this point on, we write $\Ext^{\bullet}_{(\fg,\fh)}(M, N)$ to denote $\Ext^{\bullet}_{({U}(\fg),{U}(\fh))}(M, N)$, under the conditions of the theorem.

\subsection{} The following proposition is fairly standard for Lie algebra cohomology. For relative cohomology and arbitrary modules, the authors have not found this in the 
literature. For the reader's convenience, a sketch of a proof is provided. 

\begin{prop} \label{P:ExtC} Assume the conditions of Theorem~\ref{T:relatescohom2}. For any $\fg$-modules $M$ and $N$,
\[
\Ext^\bullet_{(\fg, \fh)}(M,N) = \Ext^\bullet_{(\fg, \fh)}(\CC, \Hom_\CC(M,N)).
\]
\end{prop}

\begin{proof} We use the setup given in \cite[Theorem 5.1.1]{Lo25}. The ideas involve using a Grothendieck construction as in \cite[I 4.1 Proposition]{Ja03}, however, one has to 
make appropriate modifications to the proof given in \cite{La02}. An additional condition (cf. \cite[Section 5.1]{Lo25}) involving splitting is necessary for the spectral sequence to exist. 

Let $F:\text{Mod}(U({\mathfrak g}))\rightarrow \text{Mod}(U({\mathfrak g}))$ be the functor 
$F(-)=\text{Hom}_{{\mathbb C}}(M,-)$, and $F^{\prime}:\text{Mod}(U({\mathfrak g}))\rightarrow \text{Mod}({\mathbb C})$ be the functor 
$F^{\prime}(-)=\text{Hom}_{\mathfrak g}({\mathbb C},-)$. 

For shorthand notation, $({\mathfrak g},{\mathfrak h})$-injective will mean $(U({\mathfrak g}),U({\mathfrak h}))$-injective. Note that if $N$ is $({\mathfrak g},{\mathfrak h})$-injective then 
$\text{Hom}_{\mathbb C}(M,N)$ is $({\mathfrak g},{\mathfrak h})$-injective. Also, $F$ is $({\mathfrak g},{\mathfrak h})$-split. These results use the fact that $F(-)$ is exact because one is taking 
homomorphisms over the field ${\mathbb C}$. 

By employing \cite[Theorem 5.11]{Lo25}, one has a spectral sequence: 
\eq
E_{2}^{i,j}=\text{Ext}^{i}_{({\mathfrak g},{\mathfrak h})}({\mathbb C},R^{j}_{({\mathfrak g},{\mathfrak h})}F(N))\Rightarrow \text{Ext}^{i+j}_{({\mathfrak g},{\mathfrak h})}(M,N).
\endeq
Now $R^{j}_{({\mathfrak g},{\mathfrak h})}F(N)$ is computed by taking a $({\mathfrak g},{\mathfrak h})$-injective resolution of $N$, applying $F$ and then taking cohomology. 
Since $F$ is exact, one has $R^{j}_{({\mathfrak g},{\mathfrak h})}F(N)=0$ for $j>0$. The spectral sequence collapses and yields the result. 
\end{proof}

\subsection{Explicit complex for relative Lie superalgebra cohomology}  

We refer the reader to \cite{Ku02} for the details of many of the proofs that are stated for Lie algebras and can be modified once one accounts for the ${\mathbb Z}_2$-grading,

In this section we define (relative) Lie superalgebra cohomology for the pair $(\fg,\fh)$ where 
$\fg$ is a Lie superalgebra and $\fh \subseteq \fg$ is a Lie subsuperalgebra. This will be defined via a complex which enables us to make explicit computations. 
Later this will be related to the relative cohomology for a pair of superalgebras. 

Let $M$ be a $\fg$-module. We first define the ordinary Lie superalgebra cohomology with coefficients in $M$. 
The cochains are defined as follows. For $p\geq 0$, set 
\eq
C^{p}(\fg;M)=\Hom_{{\mathbb C}}(\Lambda^{p}_{s}(\fg),M),
\endeq
where $\Lambda^{p}_{s}(\fg)$ is the \emph{super} $p$th exterior power of $\fg$, i.e.,
the $p$-fold tensor product of $\fg$ modulo the $\fg$-submodule generated by elements of the following form: 
\eq
x_{1} \otimes \dots \otimes  x_{k} \otimes x_{k+1} \otimes \dots \otimes x_{p} + (-1)^{\p{x}_{k}\p{x}_{k+1}}
x_{1} \otimes \dots \otimes  x_{k+1} \otimes x_{k} \otimes \dots \otimes x_{p},
\endeq 
for homogeneous $x_{1}, \dots , x_{p} \in \fg.$  Thus $x_{k},x_{k+1}$ skew commute, unless both are odd in which case they commute.

Next we will define the differentials,
$d^{p}: C^{p}(\fg;M) \to C^{p+1}(\fg;M)$,
 by 
 \eq
\begin{split}\label{E:differential}
d^{p}(\phi)(x_{1}\wedge \dots \wedge x_{p+1})
&=\sum\nolimits_{i < j} (-1)^{\sigma_{i,j}(x_{1}, \dots , x_{p})} 
\phi([x_{i},x_{j}] \wedge x_{1} \wedge \dots \wedge \hat{x}_{i}\wedge \dots \wedge \hat{x}_{j}\wedge \dots  \wedge x_{p+1}) 
\\
&+ \sum\nolimits_{i}(-1)^{\gamma_{i}(x_{1}, \dots , x_{p},\phi)} 
x_{i}.\phi (x_{1} \wedge \dots \wedge \hat{x}_{i} 
\wedge \dots \wedge x_{p+1}),
\end{split}
\endeq
where $x_{1}, \dotsc , x_{p+1} \in \fg$ and $\phi \in C^{p+1}(\fg;M)$ are assumed to be homogeneous, and   
\begin{align}
\sigma_{i,j}(x_{1}, \dotsc , x_{p})
	&:=i+j+\p{x}_{i}(\p{x}_{1}+\dotsb +\p{x}_{i-1})+\p{x}_{j}(\p{x}_{1}+\dotsb +\p{x}_{j-1}+\p{x}_{i}),
\\
\gamma_{i}(x_{1}, \dotsc , x_{p}, \phi)
	&:=i+1+\p{x}_{i}(\p{x}_{1}+\dotsb + \p{x}_{i-1}+\p{\phi}).
\end{align}  
Then the ordinary cohomology is defined as  
$\text{H}^{p}(\fg, M)=\operatorname{ker} d^{p}/\operatorname{im} d^{p-1}$.

One can construct the relative cohomology as follows.  Let $\fg$, $\fh$, and $M$ be as 
above. Define 
\eq
C^{p}(\fg, \fh, M)
=\Hom_{\fh}(\Lambda^{p}_{s}(\fg/\fh), M).
\endeq
Then the map $d^{p}$ in \eqref{E:differential} descends to give a map $d^{p}: C^{p}(\fg, \fh, M) \to  C^{p+1}(\fg, \fh, M)$. Let  
${\opH}^{p}(\fg, \fh,M)=\operatorname{ker} d^{p}/\operatorname{im} d^{p-1}$. 
If $M$ is a finite-dimensional $\fg$-module and $N$ is a ${\fg}$-module then 
\eq
\Ext^p_{(\fg, \fh)}(M,N)=\opH^p(\fg, \fh, M^{*}\otimes N).
\endeq

In general the cohomology ring for an algebra is difficult to compute. However, in the Lie superalgebra setting, one can provide an explicit computation of 
$\opH^{\bullet}(\fg,\fg_{\0},{\mathbb C})$. A general computation stated below was first discovered and utilized in \cite[Theorem 2.6]{BKN10a}. This was inspired by the computation of this 
ring (as a vector space) for $\mathfrak{gl}(m|n)$ done via Kazhdan-Lusztig polynomials by Brundan (cf. \cite[Example 4.53]{Br03}). 

\begin{theorem} \label{T:(g,g_0)-coho} Let $\fg=\fg_{\0}\oplus \fg_{\1}$ be a finite-dimensional Lie superalgebra and $G_{\0}$ be an algebraic group with 
$\fg_{\0}=\operatorname{Lie }G_{\0}$. Then 
\begin{itemize}
\item[(a)] $R=\opH^{\bullet}(\fg,\fg_{\0},{\mathbb C})\cong S^{\bullet}(\fg_{\1}^{*})^{\fg_{\0}}=S^{\bullet}(\fg_{\1}^{*})^{G_{\0}}$. 
\item[(b)] If $G_{\0}$ is a reductive algebraic group then $R$ is a finitely generated ${\mathbb C}$-algebra. 
\end{itemize}
\end{theorem} 

\begin{proof} (a) First note that the second sum of \eqref{E:differential} are all zero when $M\cong {\mathbb C}$. Next observe that $[\fg_{\1},\fg_{\1}] \subseteq \fg_{\0}$, thus 
each $[x_{i}, x_{j}]$ is always zero $\fg/\fg_{\0}$, and the first sum of \eqref{E:differential} is zero. Hence,  differentials $d^{p}$ are all zero, and thus the cohomology ring 
identifies with the cochains. 

(b) The group $G_{\0}$ is reductive and acts as automorphisms on $S^{\bullet}(\fg_{\1}^{*})$. It follows by the standard invariant theory result that $R$ is finitely generated. 
\end{proof} 

\subsection{Relative Lie Superalgebra Cohomology: Spectral Sequences}  

In this section, we present a series of spectral sequences that will be useful for computing relative Lie superalgebra cohomology. 
The first spectral sequence arises from the relative cohomology for the pair $(\fg, \fh)$ with an ideal $\fii\unlhd \fg$. One should view this spectral sequence as an analog to the 
standard Lyndon-Hochschild-Serre spectral sequence. 

\begin{theorem}\label{thm:LHSSS}
Let $\fh\leq \fg$ be a Lie subsuperalgebra, $\fii \unlhd \fg$ be an ideal, and let $M$ and $N$ be $\fg$-modules. Then there exists a first quadrant 
spectral sequence 
\eq
E_{2}^{p,q}=\opExt^{p}_{(\fg/\fii, {{\mathfrak h}/{\mathfrak h}}\cap {\mathfrak i})}({\mathbb C}, 
\opExt^{q}_{({\mathfrak i},{\mathfrak h}\cap {\mathfrak i})}(M,N))\Rightarrow 
\opExt^{p+q}_{(\fg,\fh)}(M,N).
\endeq
\end{theorem}

One can replace an ideal ${\mathfrak i}$ by a subalgebra ${\mathfrak q}$, however, the relationship between the cohomology groups is more complicated. For Lie algebras, this is discussed in the seminal paper by 
Hochschild and Serre \cite{HS53}, and generalized in the work of Maurer \cite{Ma19} for Lie superalgebras. 

\begin{theorem} Let ${\mathfrak q}\leq \fh\leq \fg$ be inclusions of finite-dimensional Lie superalgebras, and let $M$ be a $\fg$-module. Then there exists a first quadrant 
spectral sequence 
$$E_{1}^{p,q}={\opH}^{q}(\fh,{\mathfrak q}, \Hom_{\mathbb C}(\Lambda^{p}(\fg/\fh),M))\Rightarrow 
{\opH}^{q}(\fg,\fq, M).$$
\end{theorem} 

In the case when $\fh=\fg_{\0}$ (a Lie algebra of a reductive group) one can use known results about relative Lie algebra cohomology to explicitly identify the $E_{2}$-term. 

\begin{theorem}[{\cite[Theorem~3.1.1]{Ma19}}] \label{thm:Maurer}
Let ${\mathfrak q}\leq \fg_{\0}\leq \fg$ be inclusions of finite-dimensional Lie superalgebras, and let $M$ be a finite-dimensional $\fg$-module. Then there exists a first quadrant 
spectral sequence 
$$E_{2}^{p,q}={\opH}^{p}(\fg,\fg_\0,M)\otimes {\opH}^{q}(\fg_\0, {\mathfrak q}, {\mathbb C}) \Rightarrow 
{\opH}^{p+q}(\fg,\fq, M).$$
\end{theorem}

\subsection{Cohomology for the pair $(\fg,\fh)$ and finite generation}\label{SS:ggzerocohom}

 We will always assume $\fg$ is a classical (quasi-reductive) Lie superalgebra.  A $\fg$-supermodule will always be assumed to be an object in the category $\mathcal{C}=\mathcal{C}_{(\fg,\fg_{\0})}$ and a finite dimensional $\fg$-supermodule will always mean an object in the category $\mathcal{F}=\mathcal{F}_{(\fg,\fg_{\0})}.$

\begin{theorem}\label{T:cohomologycalc} 
  Let $\fg = {\fg}_{\0} \oplus {\fg}_{\1}$ be a classical Lie superalgebra, and $\fa \leq {\fg}_{\0}$ an (even) subalgebra, and $M$ a $\fg$-module.
  \begin{enumerate}
    \item[(a)] There is a spectral sequence $\{E_r^{p,q}\}$ when $M$ is finite-dimensional which computes cohomology and satisfies
  \[
    E_2^{p,q}(M) \cong {\opH}^p(\g,\fg_{\0},M) \otimes {\opH}^q({\fg}_{\0},\fa,\CC) \Rightarrow {\opH}^{p+q}(\fg,\fa,M).
  \]
  For $1 \leq r \leq \infty$, $E_r^{\bullet,\bullet}(M)$ is a module for $E_2^{\bullet,\bullet}(\CC)$. 
  \item[(b)] When $M$ is finite-dimensional, ${\opH}^\bullet(\fg,\fa,M)$ is a finitely generated $R$-module.
  \item[(c)] More generally, if $C^{\bullet}({\mathfrak g},{\mathfrak g}_{\0},M)$ is a finitely generated module over $R$ then 
  ${\opH}^\bullet(\fg,\fa,M)$ is a finitely-generated $R$-module.
  \item[(d)] Moreover, the cohomology ring ${\opH}^\bullet(\fg,\fa,\CC)$ is a finitely generated $\CC$-algebra.
  \end{enumerate}
\end{theorem} 

\begin{proof}  (a) and (d) follow from \cite[Main Theorem]{Ma19}. (b) and (c) Let $M$ be a finite-dimensional $\fg$-module. Observe that $S(\fg_{\1}^{*})\otimes M$ is a finitely generated $S(\fg_{\1}^{*})$-module. Since $G_{\0}$ 
is reductive and the fixed points under $G_{\0}$ is exact, $C^{\bullet}(\fg,\fg_{\0},M)$ is a finitely generated module over $S(\fg_{\1}^{*})^{G_{\0}}$. Therefore, by Theorem~\ref{T:(g,g_0)-coho}, 
$C^{\bullet}(\fg,\fg_{\0},M)$ is a finitely generated module over $R=\opH^{\bullet}(\fg,\fg_{\0},{\mathbb C})$, thus any subquotient is finitely generated, and 
 ${\opH}^{\bullet}(\fg, \fg_{\0},M)$ is finitely generated as an $R$-module.
\end{proof} 

\subsection{Computation of  other important relative cohomology rings} 

In this section we show how to apply Theorem~\ref{T:cohomologycalc} to calculate the ring structure for cohomology rings that are relevant for the study of Category ${\mathcal O}$ for Lie superalgebras. 

\begin{theorem} \label{T:sscollapse} Let $\fg=\fg_{\0}\oplus \fg_{\1}$ be a classical Lie superalgebra, and 
${\mathfrak a}\leq \fg_{\0}$. Then there exists a spectral sequence 
$$E_{2}^{i,j}=\opH^{i}(\fg,\fg_{\0},{\mathbb C})\otimes \opH^{j}(\fg_{\0},{\mathfrak a},{\mathbb C})\Rightarrow 
\opH^{i+j}(\fg,{\mathfrak a},{\mathbb C}).$$ 
If this spectral sequence collapses then 
\begin{itemize} 
\item[(a)]  as graded rings, 
$$\opH^{\bullet}(\fg,{\mathfrak a},{\mathbb C})\cong \opH^{\bullet}(\fg,\fg_{\0},{\mathbb C})\otimes \opH^{\bullet}(\fg_{\0},{\mathfrak a},{\mathbb C}),$$ 
\item[(b)] one has homeomorphisms as topological spaces, 
$$\operatorname{Spec}(\opH^{\bullet}(\fg,{\mathfrak a},{\mathbb C}))\cong \operatorname{Spec}(\opH^{\bullet}(\fg,\fg_{\0},{\mathbb C})) 
\cong \operatorname{Spec}(S^{\bullet}(\fg_{\1})^{G_{\0}}).$$ 
\end{itemize} 
\end{theorem}  

\begin{proof} (a) Set $R=\opH^{\bullet}(\fg,\fg_{\0},{\mathbb C})\cong S^{\bullet}(\fg_{\1})^{G_{\0}}$. 
First observe that if the spectral sequence collapses then one has an isomorphism of $R$-modules: 
\eq
\opH^{\bullet}(\fg,{\mathfrak a},{\mathbb C})\cong \opH^{\bullet}(\fg,\fg_{\0},{\mathbb C})\otimes \opH^{\bullet}(\fg_{\0},{\mathfrak a},{\mathbb C}).
\endeq
Furthermore, the ring $R$ identifies as a subring of $\opH^{\bullet}(\fg,\fg_{\0},{\mathbb C})\otimes \opH^{\bullet}(\fg_{\0},{\mathfrak a},{\mathbb C})$. 

In order to prove (a), we need to show that $\opH^{\bullet}(\fg_{\0},{\mathfrak a},{\mathbb C})$ can also be identified as a subring. One can then apply the argument in 
\cite[Section 3.1]{DNN12} to show that the cohomology ring identifies as the tensor product of rings. 
 
The cohomology $\opH^{\bullet}(\fg,{\mathfrak a},{\mathbb C})$ is obtained by the cochains 
\eq
C^{\bullet}(\fg,{\mathfrak a},{\mathbb C})=\Hom_{\mathfrak a}(\Lambda^{\bullet}(\fg/{\mathfrak a}),{\mathbb C}).
\endeq
Furthermore, using the notation in \cite[3.2]{Ma19}, we have a decomposition for $n\geq 0$: 
\eq
C^{n}(\fg,{\mathfrak a},{\mathbb C})=C^{n}(\fg,{\mathfrak a},{\mathbb C})_{(1)}\oplus C^{n}(\fg_{\0},{\mathfrak a},{\mathbb C}).
\endeq
where $C^{\bullet}(\fg,{\mathfrak a},{\mathbb C})_{(1)}$ and $C^{\bullet}(\fg_{\0},{\mathfrak a},{\mathbb C})$ are subcomplexes of 
$C^{n}(\fg,{\mathfrak a},{\mathbb C})$. This shows that $\opH^{\bullet}(\fg_{\0},{\mathfrak a},{\mathbb C})$ is a subring of 
$\opH^{\bullet}(\fg,{\mathfrak a},{\mathbb C})$ with $\opH^{\bullet}(\fg,{\mathfrak a},{\mathbb C})/F^{1}\opH^{\bullet}(\fg,{\mathfrak a},{\mathbb C})
\cong \opH^{\bullet}(\fg_{\0},{\mathfrak a},{\mathbb C})$.

(b) This follows immediately from (a) since $\opH^{\bullet}(\fg_{\0},{\mathfrak a},{\mathbb C})$ is finite-dimensional. Therefore, the elements in positive degree for this ring 
are nilpotent and contained in the radical of $\opH^{\bullet}(\fg,{\mathfrak a},{\mathbb C})$. 
\end{proof} 

As a application of the preceding theorem, one can use facts about the Kazhdan-Lusztig theory in the parabolic Category ${\mathcal O}$ for 
semisimple Lie algebras in order to compute other cohomology rings. 

\begin{cor} \label{C:cohoringcalc} Let $\fg$ be a classical simple Lie superalgebra where ${\mathfrak a}$ 
is one of the following subalgebras in the reductive Lie algebra $\fg_{\0}$:
\begin{itemize} 
\item[(a)] ${\mathfrak a}=\ft_{\0}$ is a maximal torus in $\fg_{\0}$;
\item[(b)] ${\mathfrak a}={\mathfrak l}_{\0}$ is a Levi subalgebra in $\fg_{\0}$;
\item[(c)] ${\mathfrak a}=\fb_{\0}$ is a Borel subalgebra in $\fg_{\0}$;
\item[(d)] ${\mathfrak a}={\mathfrak p}_{\0}$ is a parabolic subalgebra in $\fg_{\0}$.
\end{itemize} 
Then,
$$\opH^{\bullet}(\fg,{\mathfrak a},{\mathbb C})\cong \opH^{\bullet}(\fg,\fg_{\0},{\mathbb C})\otimes \opH^{\bullet}(\fg_{\0},{\mathfrak a},{\mathbb C})$$ 
as graded rings, and 
$$\operatorname{Spec}(\opH^{\bullet}(\fg,{\mathfrak a},{\mathbb C}))\cong \operatorname{Spec}(\opH^{\bullet}(\fg,\fg_{\0},{\mathbb C})) 
\cong \operatorname{Spec}(S^{\bullet}(\fg_{\1})^{G_{\0}})$$ 
as topological spaces. 
\end{cor} 

\begin{proof} It suffices to prove (b) and (d). (b) Note that $\opH^{\bullet}(\fg_{\0},{\mathfrak l}_{\0},{\mathbb C})$ is non-zero only in even degrees by looking at the cohomology of the trivial module 
for the parabolic Category ${\mathcal O}$ for the reductive Lie algebra $\fg_{\0}$. Therefore, the differentials in the spectral sequence 
Theorem~\ref{T:sscollapse} are equal to zero, and the spectral sequence collapses. This yields the result. 

(d) By using the preceding argument for (b), it suffices to show that $\opH^{\bullet}(\fg_{\0},{\mathfrak p}_{\0},{\mathbb C})$ is non-zero only in even degrees. 
Let ${\mathfrak l}_{\0}$ be the Levi subalgebra for ${\mathfrak p}_{\0}$ with ${\mathfrak p}_{\0}\cong {\mathfrak l}_{\0}\oplus {\mathfrak u}^{+}$. One has ${\mathfrak l}_{\0}\leq {\mathfrak p}_{\0} \leq \fg_{\0}$, 
and there exists 
a spectral sequence: 
\eq
E_{1}^{p,q}=\opH^{q}({\mathfrak p}_{\0},{\mathfrak l}_{\0},\Lambda^{p}((\fg_{\0}/{\mathfrak p}_{\0})^{*}))\Rightarrow 
\opH^{p+q}(\fg_{\0},{\mathfrak p}_{\0},{\mathbb C}).
\endeq
Then 
\eq
\begin{split} 
\opH^{q}({\mathfrak p}_{\0},{\mathfrak l}_{\0},\Lambda^{p}_{\0}((\fg_{\0}/{\mathfrak p}_{\0})^{*}))
	&\cong
	\operatorname{Ext}^{q}_{({\mathfrak p}_{\0},{\mathfrak l}_{\0})}({\mathbb C},\Lambda^{p}((\fg_{\0}/{\mathfrak p}_{\0})^{*}))
	\\
&\cong
	\operatorname{Ext}^{q}_{({\mathfrak p}_{\0},{\mathfrak l}_{\0})}(\Lambda^{p}(\fg_{\0}/{\mathfrak p}_{\0}),{\mathbb C})
	\\	
&\cong
	\operatorname{Ext}^{q}_{(\fg_{\0},{\mathfrak l}_{\0})}(U(\fg_{\0})\otimes_{U({\mathfrak p}_{\0})} \Lambda^{p}(\fg_{\0}/{\mathfrak p}_{\0}),{\mathbb C}).
\end{split} 
\endeq
The module $U(\fg_{\0})\otimes_{U({\mathfrak p}_{\0})} \Lambda^{p}(\fg_{\0}/{\mathfrak p}_{\0})$ has a filtration of Verma modules 
$Z_{\0}(\sigma)$ where $\sigma$ is a weight of $\Lambda^{p}(\fg_{\0}/{\mathfrak p}_{\0})$. The only factors for which 
$\operatorname{Ext}^{q}_{(\fg_{\0},{\mathfrak l}_{\0})}(Z_{\0}(\sigma),{\mathbb C})\neq 0$ is when $\sigma\in W\cdot 0$. Note that  
$\sigma=w\cdot 0$ implies that $l(w)=p$, and by the validity of the Kazhdan-Lusztig conjectures, $q$ and $l(w)=p$ have the same parity. 
Consequently, $p+q$ must be even for the terms in the spectral sequence to be non-zero. This proves that 
$\opH^{\bullet}(\fg_{\0},{\mathfrak p}_{\0},{\mathbb C})$ is non-zero only in even degrees. 
\end{proof} 

\section{Relating the categorical cohomology to the relative Lie algebra cohomology}

\subsection{} Let ${\mathfrak g}={\mathfrak n}^{-}\oplus {\mathfrak a} \oplus {\mathfrak n}^{+}$ where ${\mathfrak p}={\mathfrak a}\oplus {\mathfrak n}^{+}$. Moreover, 
let $\cO^\fp$ be the Category ${\mathcal O}$ for $U({\mathfrak g})$-modules as defined in Definition~\ref{def:O}, and $\cO^{\fp_{\0}}$ be the Category ${\mathcal O}$ for $U({\mathfrak g}_{\0})$-modules relative to the 
decomposition ${\mathfrak g}_{\0}={\mathfrak n}^{-}_{\0}\oplus {\mathfrak a}_{\0} \oplus {\mathfrak n}^{+}_{\0}$ where ${\mathfrak p}_{\0}={\mathfrak a}_{\0}\oplus {\mathfrak n}^{+}_{\0}$.
We first show that modules in $\cO^\fp$ restrict to $U({\mathfrak g}_{\0})$-modules in $\cO^{\fp_{\0}}$. 

\begin{lemma} \label{L:ptop0}
If $M \in \cO^\fp$, then $M \in \cO^{\fp_{\0}}$.
\end{lemma}
\begin{proof} We employ Definition~\ref{def:O}. Let ${\mathfrak g}={\mathfrak n}^{-}\oplus {\mathfrak l} \oplus {\mathfrak n}^{+}$, ${\mathfrak g}_{\0}={\mathfrak n}^{-}_{\0}\oplus {\mathfrak l}_{\0} \oplus {\mathfrak n}^{+}_{\0}$. 
and let $M\in \cO^\fp$. Since $M$ is finitely generated over $U({\mathfrak g})$, $M$ is the homomorphic image of $\bigoplus_{i=1}^{n} U({\mathfrak g})$. As a left $U({\mathfrak g}_{\0})$-module, one 
has 
\eq
\bigoplus\nolimits_{i=1}^{n} U({\mathfrak g})
\cong \bigoplus\nolimits_{i=1}^{n} U({\mathfrak g}_{\0})\otimes U({\mathfrak g}_{\1})
\cong \bigoplus\nolimits_{j=1}^{m}U({\mathfrak g}_{\0}).
\endeq
This shows that $M$ is finitely generated as $U({\mathfrak g}_{\0})$-module. 

Next as a $U({\mathfrak l})$-module, $M\cong \bigoplus_{i=1}^{s} M_{i}$ where $M_{i}$ are finite-dimensional $U({\mathfrak l})$-modules. Each $M_{i}$ restricted to $U({\mathfrak l}_{\0})$ is a direct 
sum of finite-dimensional $U({\mathfrak l}_{\0})$-modules, which also is true for $M$. 

Finally, let $m\in M$. Since $M\in \cO^\fp$, one has $\dim U({\mathfrak n}^{+}).m <\infty$, thus $\dim U({\mathfrak n}^{+}_{\0}).m < \infty$. It follows that 
$M$ is locally ${\mathfrak n}^{+}_{\0}$-finite. Hence, $M \in \cO^{\fp_{\0}}$.
\end{proof} 

Since modules in  $\cO^{\fp_{\0}}$ have finitely many composition factors by Proposition~\ref{Cat O-Liealg prop}(a), one can state, justified by the preceding lemma, the following theorem about composition factors in $\cO^\fp$. 

\begin{theorem}\label{T:finitecompfactors} Let $V \in \cO^\fp$ where ${\mathfrak p}$ is a principal parabolic subalgebra. Then $V$ has only finitely many composition factors. 
\end{theorem}

\subsection{Comparison between ${\mathcal O}^{\mathfrak p}$ and ${\mathcal O}^{{\mathfrak p}_{\0}}$} \label{SS:Comparison} 

Let $\Lambda$ index the simple ${\mathcal O}^{\mathfrak p}$-modules, and for $\lambda\in \Lambda$, let 
$L(\lambda)$ be the corresponding simple module. Moreover, let $P(\lambda)$ be the projective cover of $L(\lambda)$. 
Similarly, let $\Lambda_{\0}$ index the simple modules in ${\mathcal O}^{{\mathfrak p}_{\0}}$, and for $\sigma\in \Lambda_{\0}$, let 
$L_{\0}(\sigma)$ be the corresponding simple module, and $P_{\0}(\sigma)$ be its projective cover.  The next proposition shows that there are enough projectives in 
${\mathcal O}^{\mathfrak p}$. 

\begin{prop}\label{P:projprop} Let $\lambda\in \Lambda$ and $\sigma\in \Lambda_{\0}$ with ${\mathfrak p}$ a principal parabolic of ${\mathfrak g}$. 
\begin{itemize} 
\item[(a)] $U({\mathfrak g})\otimes_{U({\mathfrak g}_{\0})}P_{\0}(\sigma)$ is a projective module in ${\mathcal O}^{\mathfrak p}$. 
\item[(b)] $(U({\mathfrak g})\otimes_{U({\mathfrak g}_{\0})}P_{\0}(\sigma):P(\lambda))=[L(\lambda):L_{\0}(\sigma)]$. 
\end{itemize}
In part (b), the left hand side of the equation measures how many times a projective summand occurs, and the right hand side is the $U({\mathfrak g}_{\0})$ composition factor multiplicity. 
Note the composition multiplicity is finite since $L(\lambda)$ upon restriction to $U({\mathfrak g}_{\0})$ is in ${\mathcal O}^{{\mathfrak p}_{\0}}$. 
\end{prop} 

\begin{proof} (a) Consider the functor $\text{Hom}_{{\mathcal O}^{\mathfrak p}}(U({\mathfrak g})\otimes_{U({\mathfrak g}_{\0})}P_{\0}(\sigma),-)$ on objects in 
${\mathcal O}^{\mathfrak p}$. One has the following natural isomorphisms: 
\eq
\begin{split} 
\text{Hom}_{{\mathcal O}^{\mathfrak p}}(U({\mathfrak g})\otimes_{U({\mathfrak g}_{\0})}P_{\0}(\sigma),-)&\cong 
\text{Hom}_{U({\mathfrak g})}(U({\mathfrak g})\otimes_{U({\mathfrak g}_{\0})}P_{\0}(\sigma),-)
\\
&\cong \text{Hom}_{U({\mathfrak g}_{\0})}(P_{\0}(\sigma),-)
\\
&\cong \text{Hom}_{{\mathcal O}^{{\mathfrak p}_{\0}}}(P_{\0}(\sigma),-). 
\end{split} 
\endeq
Since $P_{\0}(\sigma)$ is projective in ${\mathcal O}^{{\mathfrak p}_{\0}}$, it follows that  $\text{Hom}_{{\mathcal O}^{\mathfrak p}}(U({\mathfrak g})\otimes_{U({\mathfrak g}_{\0})}P_{\0}(\sigma),-)$ is 
exact, and $U({\mathfrak g})\otimes_{U({\mathfrak g}_{\0})}P_{\0}(\sigma)$ is a projective module in ${\mathcal O}^{\mathfrak p}$.

(b) One can relate these multiplicities by using the dimensions of Hom-spaces: 
\begin{equation}
\begin{split}
(U({\mathfrak g})\otimes_{U({\mathfrak g}_{\0})}P_{\0}(\sigma):P(\lambda))
&=  \dim \text{Hom}_{{\mathcal O}^{\mathfrak p}}(U({\mathfrak g})\otimes_{U({\mathfrak g}_{\0})}P_{\0}(\sigma),L(\lambda)) 
\\
&=  \dim \text{Hom}_{U({\mathfrak g})}(U({\mathfrak g})\otimes_{U({\mathfrak g}_{\0})}P_{\0}(\sigma),L(\lambda)) 
\\
&=\dim \text{Hom}_{U({\mathfrak g}_{\0})}(P_{\0}(\sigma),L(\lambda)) 
\\ 
&= \dim \text{Hom}_{{\mathcal O}^{{\mathfrak p}_{\0}}}(P_{\0}(\sigma),L(\lambda)) 
\\ 
&= [L(\lambda):L_{\0}(\sigma)]. \qedhere
\end{split}
\end{equation}
\end{proof} 

A similar argument shows that there are enough injectives in ${\mathcal O}^{\mathfrak p}$. 

\subsection{A Spectral Sequence}\label{SS:spectralseqcat} Consider the left exact functor 
\eq\label{def:F}
\cF: \cC_{(\fg, \fa_\0)} \to \cO^\fp,
\quad M \mapsto \textup{largest submodule of }M\textup{ in }\cO^\fp.
\endeq 

\begin{prop} \label{P:split} Assume that  for all finite-dimensional $U({\mathfrak a}_{\0})$-modules $M$ and $N$, $\operatorname{Ext}^{1}_{U({\mathfrak a}_{\0})}(M,N)=0$. Then the functor ${\mathcal F}$ takes relative injective objects in $\cC_{(\fg, \fa_\0)}$ to injective objects in $\cO^\fp$. 
\end{prop} 

\begin{proof} A relative injective in $\cC_{(\fg, \fa_\0)}$ is a summand of $I=\text{Hom}_{U({\mathfrak a}_{\0})}(U({\mathfrak g}),T)$ where $T$ is a finite-dimensional ${\mathfrak a}_{\0}$-module. So it 
suffices to show that ${\mathcal F}(I)$ is injective in $\cO^\fp$, or equivalently $\text{Hom}_{\cO^\fp}(-,{\mathcal F}(I))$ is exact. From using adjointness and Frobenius reciprocity, one has  
\eq
\text{Hom}_{\cO^\fp}(-,{\mathcal F}(I))\cong \text{Hom}_{U({\mathfrak g})}(-,I)\cong \text{Hom}_{U({\mathfrak a}_{0})}(-,T).
\endeq
Now the exactness follows using the hypothesis, since any object in $\cO^\fp$ is a direct sum of finite-dimensional $U({\mathfrak a}_{\0})$-modules and $T$ is a finite-dimensional $U({\mathfrak a}_{\0})$-module. 
\end{proof}

The preceding proposition enables us to construct a spectral sequence that relates categorical cohomology in $\cO^\fp$ to 
relative $({\mathfrak g},{\mathfrak a}_{\0})$ cohomology. 

\begin{theorem} \label{thm:E2OpRj} Let ${\mathfrak a}$ satisfy the conditions of Proposition~\ref{P:split}. Then there exists a spectral sequence 
\begin{equation} \label{E:spectralCatO}
E_2^{i,j} = \Ext^i_{\cO^\fp}(M, R^j\cF(N)) \Rightarrow \Ext^{i+j}_{({\mathfrak g},{\mathfrak a}_{\0})}(M,N)
\end{equation}
where $M, N\in \cC_{(\fg, \fa_\0)}$.
\end{theorem} 

\begin{proof}  We will proceed as in Proposition~\ref{P:ExtC}.  Let $F:\cC_{(\fg, \fa_\0)}\rightarrow \cO^\fp$ be the functor 
$F(-)=\text{Hom}_{\cO^\fp}(M,-)$, and $F^{\prime}:\cC_{(\fg, \fa_\0)}\rightarrow \cO^\fp$ be the functor 
$F^{\prime}(-)={\mathcal F}(-)$. 

Using the assumption, in all the categories involved, one has that the morphisms are ${\mathfrak a}_{\0}$-split, so the conditions of \cite[Theorem 5.1.1]{Lo25} hold. Consequently, one has a 
spectral sequence  
\eq
E_{2}^{i,j}=R^{i}_{\cO^\fp}F^{\prime}(R^{j}_{({\mathfrak g},{\mathfrak a}_{\0})}F(N))\Rightarrow \text{Ext}^{i+j}_{({\mathfrak g},{\mathfrak a}_{\0})}(M,N)
\endeq
which proves the result.  
\end{proof} 

Let $\Lambda$ (resp. $\Lambda_{\0}$) with the indexing of the corresponding simple and projective modules be as in Section~\ref{SS:Comparison}. For $\lambda\in \Lambda$, $j\geq 0$ and $N\in 
\cC_{(\fg, \fa_\0)}$,  
\begin{equation} 
[R^j\cF(N): L(\lambda)] = \dim \Ext^j_{(\fg,\fa_\0)}(P(\lambda), N). 
\end{equation}
This follows directly from the spectral sequence (\ref{E:spectralCatO}) by setting $M=P(\lambda)$ and observing that the spectral sequence collapses to yield 
$\text{Hom}_{\cO^\fp}(P(\lambda), R^j\cF(N))\cong \Ext^{j}_{({\mathfrak g},{\mathfrak p}_{\0})}(P(\lambda),N)$.

\subsection{Proof of Theorem~\ref{T:cohomology}}

We begin by proving that the higher right derived functor of ${\mathcal F}$ vanish on objects in $\cO^\fp$. 
\prop\label{prop:RjF0}
For $j>0$, $R^j\cF(N) = 0$ for $N \in \cO^\fp$ where ${\mathfrak p}$ is principal. 
\endprop
\proof

Let $M= P(\lambda)$ be the projective cover of $L(\lambda)$ in $\cO^\fp$. The spectral sequence collapses, and thus for $n\geq 0$,
\eq \label{E:hom-extiso}
\Hom_{\cO^\fp}(P(\lambda), R^j\cF(N)) \cong \Ext^j_{(\fg,\fl_\0)}(P(\lambda), N) .
\endeq
It remains to show that $\Ext^j_{(\fg,\fl_\0)}(P(\lambda), N) = 0$ for $j \geq 0$ and $\lambda\in \Lambda$.

First, from Proposition~\ref{P:projprop}, $U(\fg)\otimes_{U(\fg_\0)} P_{\0}(\sigma)$ is projective in ${\mathcal O}^{\mathfrak p}$ for all $\sigma\in \Lambda_{\0}$. Moreover, 
for some $\sigma\in \Lambda_{\0}$, $P(\lambda)$ is a direct summand of  $U(\fg)\otimes_{U(\fg_\0)} P_{\0}(\sigma)$. For $j\geq 0$, applying (\ref{E:hom-extiso}) 
\eq
\begin{split} 
[R^j\cF(N): L(\lambda)] &= \dim \Ext^j_{(\fg,\fl_\0)}(P(\lambda), N) 
\\
&\leq \dim \Ext^j_{(\fg,\fl_\0)}(U(\fg) \otimes_{U(\fg_\0)} P_{\0}(\sigma), N)
\\
&=\dim \Ext^j_{(\fg_{\0},\fl_\0)}( P_{\0}(\sigma), N)
\\
&=\dim \Ext^j_{{\mathcal O}^{{\mathfrak p}_{\0}}}( P_{\0}(\sigma), N)
= 0.
\end{split} 
\endeq
Therefore, $R^j\cF(N)=0$ for $j>0$. 
\endproof

We can now present a proof of Theorem~\ref{T:cohomology}. 

\begin{proof} For $M,N \in \cO^\fp$, one has a spectral sequence
\eq
E_2^{i,j} = \Ext^i_{\cO^\fp}(M, R^j\cF(N)) \Rightarrow \Ext^{i+j}_{(\fg, \fl_\0)}(M,N).
\endeq
By Proposition~\ref{prop:RjF0}, $R^j\cF(N)=0$ for $j>0$. Therefore, the spectral sequence collapses, and for $n\geq 0$; 
\eq
\Ext^n_{(\fg, \fl_\0)}(M,N) \cong \Ext^n_{\cO^\fp} (M, \cF(N)) \cong \Ext^n_{\cO^\fp}(M,N).\qedhere
\endeq
\end{proof}

\subsection{Proof of Theorem~\ref{T:finite generation}} \label{sec:proofB}
For (a), first note by Theorem~\ref{T:cohomology}, one has $R=\Ext^{\bullet}_{{\cO^\fp}}(\mathbb{C},\mathbb{C})\cong \opH^{\bullet}({\mathfrak g},{\mathfrak l}_{\0},\mathbb C)$. Now by 
Corollary~\ref{C:cohoringcalc}(b), 
\eq
\opH^{\bullet}({\mathfrak g},{\mathfrak l}_{\0},\mathbb C)\cong S^{\bullet}({\mathfrak g}_{\1}^{*})^{G_{\0}}\otimes \opH^{\bullet}(\fg_{\0},{\mathfrak l}_{\0},{\mathbb C}).
\endeq
For (b), it follows by combining Theorem~\ref{T:cohomology} and Theorem~\ref{T:cohomologycalc}(b). 
For (c), first observe that if 
$M$ is finite-dimensional and $N$ is in $\cO^\fp$ then 
\eq
\Ext^{\bullet}_{\cO^\fp}(M,N)\cong \Ext^{\bullet}_{\cO^\fp}(\mathbb C,M^{*}\otimes N)
\endeq 
with $M^{*}\otimes N$ in $\cO^\fp$.  According to Theorem~\ref{T:cohomologycalc}(c), it suffices to prove that  
$C^{\bullet}({\mathfrak g},{\mathfrak g}_{\0},\widehat{N})$ is a finitely generated module over $R=S^{\bullet}({\mathfrak g}_{\1}^{*})^{G_{\0}}$ 
where $\widehat{N}$ is in $\cO^\fp$. For ${\mathfrak g}$ classical simple, it was shown in \cite[Theorem 3.4]{BKN10a} that as an $R$-module one has a harmonic decomposition 
\begin{equation}
S^{\bullet}({\mathfrak g}_{\1}^{*})\cong S^{\bullet}({\mathfrak g}_{\1}^{*})^{G_{\0}} \otimes [\text{ind}_{H}^{G_{\0}} {\mathbb C} ]_{\bullet}. 
\end{equation} 
Therefore, as $R$-modules, 
\begin{equation} 
C^{\bullet}({\mathfrak g},{\mathfrak g}_{\0},M)\cong \text{Hom}_{U({\mathfrak g}_{\0})}(S^{\bullet}({\mathfrak g}_{\1}),\widehat{N})\cong 
S^{\bullet}({\mathfrak g}_{\1}^{*})^{G_{\0}}\otimes \text{Hom}_{U({\mathfrak g}_{\0})}({\mathbb C},[\text{ind}_{H}^{G_{\0}} {\mathbb C}]_{\bullet} \otimes \widehat{N}). 
\end{equation} 
It suffices to show that $\text{Hom}_{U({\mathfrak g}_{\0})}({\mathbb C},[\text{ind}_{H}^{G_{\0}} {\mathbb C }] \otimes \widehat{N})$ is finite-dimensional. 

Set ${\mathfrak h}=\text{Lie }H$. Observe that $\text{ind}_{H}^{G_{\0}} {\mathbb C }$ is locally finite induction, so there exists an injective homomorphism 
$\text{ind}_{H}^{G_{\0}} {\mathbb C }\hookrightarrow \text{Hom}_{U({\mathfrak h})}(U({\mathfrak g}),{\mathbb C})$. This induces an injective map 
\eq
\text{Hom}_{U({\mathfrak g}_{\0})}({\mathbb C},[\text{ind}_{H}^{G_{\0}} {\mathbb C }] \otimes \widehat{N})\hookrightarrow 
 \text{Hom}_{U({\mathfrak g}_{\0})}({\mathbb C}, [\text{Hom}_{U({\mathfrak h})}(U({\mathfrak g}),{\mathbb C})]\otimes \widehat{N}).
 \endeq
 Now by Frobenius reciprocity, 
 \eq
 \text{Hom}_{U({\mathfrak g}_{\0})}({\mathbb C}, [\text{Hom}_{U({\mathfrak h})}(U({\mathfrak g}),{\mathbb C})]\otimes \widehat{N})
 \cong \text{Hom}_{U({\mathfrak h})}({\mathbb C}, \widehat{N}).
 \endeq
Let $\widehat{T}$ be a torus in $H$. Since $\widehat{N}$ is a rational $T$-module, one has an injective map 
\eq
\text{Hom}_{U({\mathfrak h})}({\mathbb C}, \widehat{N})\hookrightarrow \text{Hom}_{U(\hat{\mathfrak t})}({\mathbb C}, \widehat{N})\cong 
\text{Hom}_{\widehat{T}}({\mathbb C}, \widehat{N}).
\endeq 
With this series of injective maps, one is reduced to showing that $\text{Hom}_{\widehat{T}}({\mathbb C}, \widehat{N})$ is finite-dimensional. 

Since $\widehat{N}$ has a finite composition series in $\cO^\fp$ and in $\cO^{\fp_{\0}}$, we can assume without loss of generality that 
$\widehat{N}$ is simple, or more generally that $\widehat{N}$ is quotient of a parabolic Verma module in $\cO^{\fp_{\0}}$. The weights of 
$\widehat{N}$ are in the cone $\lambda+{\mathbb N}\Phi_{\0}^{-}$ (with finite multiplicities) for some weight $\lambda$. Hence, we are reduced to showing that 
for any given $T$-weight $\lambda$, there are finitely many $T$-weights in ${\mathbb N}\Phi_{\0}^{+}$ that restrict to $\lambda$ under $\widehat{T}$. Note that this is true, if $\widehat{T}=T$ is a maximal torus of $G_{\0}$, 
but here $\widehat{T}$ can be a smaller torus. 

The proof is finished off by using the data in the Appendix, where the argument proceeds in a case by case manner. For each almost simple Lie algebra (see Section \ref{sec:almostsimple}) the torus $\widehat{T}$ is identified, and a grading on the 
roots is specified, in order to verify the aforementioned claim about finite-dimensionality. 

\section{Standard Stratification of $\cO^\fp$}

One of the main ideas of this section is to extend work of Holmes and Nakano \cite{HN91} for finite-dimensional graded algebras and their finite-dimensional modules to the setting of infinite-dimensional 
modules in ${\mathcal O}^{\mathfrak p}$. We will indicate places where several of the results in \cite{HN91} can be used. 

In the past, many of the approaches to proving that the categories for Lie superalgebras are stratified were done on a case by case basis. In the following sections, we provide 
a uniform treatment to prove stratification for a wide array of examples. 

 \subsection{Ext Orthogonality}
Suppose that $\fp$ is a principal parabolic subalgebra corresponding to the principal parabolic subset $P_\cH$ for some functional $\cH\in V^*$.
Then, $\fp = \fa \oplus \fn^+$, where $\fa = \bigoplus_{\alpha \in \Phi^0_\cH} {\mathfrak g}_{\alpha}$ is the corresponding ``Levi subalgebra".
 
Consider the category ${\mathcal C}_{({\mathfrak a},{\mathfrak a}_{\0})}$. Since $\fp$ is a principal parabolic subalgebra, ${\mathfrak a}_{\0}=
{\mathfrak l}_{\0}$ where ${\mathfrak l}_{\0}={\mathfrak l}_{J}$ is a Levi subalgebra for ${\mathfrak g}_{\0}$. The finite-dimensional modules for ${\mathfrak l}_{\0}$ are indexed by 
dominant integral weights $X_{J,+}$ (i.e., dominant weights on, $J$, a subset of simple roots). 

For each $\sigma \in X_{J,+}$, set $M(\sigma)=U({\mathfrak a})\otimes_{U({\mathfrak a}_{\0})} L_{J}(\sigma)$. For a simple ${\mathcal C}_{({\mathfrak a},{\mathfrak a}_{\0})}$-module, 
$S$, let $P(S)$ be its projective cover. Then $P(S)$ is a direct summand of $M(\sigma)$ for some $\sigma$ and $\cH(\gamma)=\cH(\sigma)$ for any weight $\gamma$ in $M(\sigma)$, 
thus any weight of $S$. Note a similar argument shows that if $I(S)$ is the injective hull of $S$ then $\cH(\gamma)$ is constant for all weights $\gamma$ of $I(S)$. 

Let $\Lambda_{\mathfrak a}$ index the non-isomorphic simple $\fa$-modules in $\cC_{(\fa, \fa_\0)}$. For $\lambda_{0}\in \Lambda_{\mathfrak a}$, let $L_{\mathfrak a}(\lambda_{0})$ be the 
corresponding simple module and $I_{\mathfrak a}(\lambda_{0})$ be its injective hull.  Any simple module in $\cO^\fp$ can be realized as a quotient of $U({\mathfrak g})\otimes _{U({\mathfrak p})} L_{\mathfrak a}(\lambda_{0})$ for some 
 $\lambda_{0}\in \Lambda_{\mathfrak a}$. One can now use the argument in \cite[Section 3]{HN91} to show that the simples in $\cO^\fp$ are also indexed by $\Lambda_{{\mathfrak a}}$ (before considering 
 parity). Set $\Lambda:= \Lambda_{\mathfrak a} \times \{0,1\}$. For any $\lambda=(\lambda_{0},i)\in \Lambda$, let $L(\lambda)$ be the corresponding simple module in $\cO^\fp$. 
 Here it is important to remember that $\lambda_{0}$ and $\lambda$ should not be regarded as weights.  

For any $\sigma\in {\mathfrak t}^{*}$ where ${\mathfrak t}$ is a maximal torus in ${\mathfrak a}_{\0}$, let $\overline{\sigma}:= \cH(\sigma)$. Then, ${\mathfrak t}^{*}$ has a proset structure given by
\eq
\sigma_{1} \leq \sigma_{2} \in {\mathfrak t}^{*} 
\quad \Leftrightarrow \quad
\overline{\sigma}_{1} \leq \overline{\sigma}_{2} \in \RR.
\endeq
Now one can put a proset structure on $\Lambda$ as follows. Let $\mu, \lambda \in \Lambda$ and $\sigma_{1}, \sigma_{2}$ be weights such that 
$I_{\mathfrak a}(\mu_{0})$ (resp. $I_{\mathfrak a}(\lambda_{0})$) has weights $\gamma$ such that $\overline{\gamma}=\overline{\sigma}_{1}$ (resp. $\overline{\gamma}=\overline{\sigma}_{2}$). 
Then $\Lambda:= \Lambda_0 \times \{0,1\}$ is a proset via
\eq\label{eq:Ofaposet}
(\mu_{0}, i) \leq (\lambda_{0}, j) \in \Lambda \quad \Leftrightarrow \quad i=j, \ \sigma_{1}  \leq \sigma_{2} \quad \Leftrightarrow  \quad i=j,\ \overline{\sigma}_{1}  \leq \overline{\sigma}_{2}.
\endeq

Recall the functor $\cF$ from \eqref{def:F}.
For $\lambda =(\lambda_0, i)\in \Lambda$, 
denote the (co)induced modules\footnote{The construction of the intermediate modules of this type first appeared in the second author's Ph.D thesis in 1990 (cf. \cite[p.7]{N92}). 
These modules were used to study the representation theory of Lie algebras of Cartan type.} from relative injectives and irreducibles, respectively, by 
 \eq\label{eq:ninjDirr}
\ninj(\lambda) := \Pi^i \cF(\Hom_{U(\fp^-)} (U(\fg), I_\fa(\lambda_0))),
\quad
\Dirr(\lambda) := \Pi^i U(\fg) \otimes_{U(\fp)} L_\fa(\lambda_0).
 \endeq
 
We start by establishing an important Ext-orthogonality property between these modules.

\prop\label{prop:exto}
Let $\lambda, \mu \in \Lambda$. Then,
\[
\Ext^n_{\cO^\fp}(\Dirr(\lambda), \ninj(\mu)) = \begin{cases} 
\CC &\tif n=0,\  \lambda = \mu;
\\
0 &\textup{otherwise}.
\end{cases}
\]
\endlem
\proof
Let $M:= \Hom_{U(\fp^-)} (U(\fg), I_\fa(\mu_{0}))$.
Consider the spectral sequence from Theorem~\ref{thm:E2OpRj}:
\eq\label{eq:EOSS1}
E_2^{i,j} = \Ext^i_{\cO^\fp}(\Dirr(\lambda), R^j\cF(M)) \Rightarrow \Ext^{i+j}_{(\fg, \fa_\0)}(\Dirr(\lambda), M).
\endeq
We want to show that $R^j\cF(M) = 0$ for $j > 0$ (note that $M$ does not necessarily lie in $\cO^\fp$ so Proposition~\ref{prop:RjF0} does not apply).
If this happens, then the spectral sequence \eqref{eq:EOSS1} collapses, and thus
\eq
\begin{split}
\Ext^n_{\cO^\fp}(\Dirr(\lambda), \ninj(\mu))  &\cong \Ext^n_{(\fg, \fa_\0)}(\Dirr(\lambda), M)
\\
&\cong \Ext^n_{(\fp^-, \fa_\0)}(\Dirr(\lambda),  I_\fa(\mu_{0})) \quad \textup{by adjointness }
\\
&\cong \Ext^n_{(\fa, \fa_\0)}(L_\fa(\lambda_{0}), I_\fa(\mu_{0})) \quad \textup{since } \Dirr(\lambda)|_{\fp^-} \cong U(\fp^-) \otimes_{U(\fa)} L_\fa(\lambda_0).
\end{split}
\endeq
That is, the Ext-orthogonality result holds.

Next, in order to show that $R^j\cF(M) = 0$ for $j>0$, we consider the following spectral sequence, again from Theorem~\ref{thm:E2OpRj}:
\eq\label{eq:EOSS2}
E_2^{i,j} = \Ext^i_{\cO^\fp}(P(\lambda), R^j\cF(M)) \Rightarrow \Ext^{i+j}_{(\fg, \fa_\0)}(P(\lambda), M).
\endeq
The spectral sequence \eqref{eq:EOSS2} collapses since $P(\lambda)$ is projective in $\cO^\fp$, and hence
\eq
\Hom_{\cO^\fp}(P(\lambda), R^j\cF(M)) \cong \Ext^j_{(\fg,\fa_\0)} (P(\lambda), M).
\endeq
One is now reduced to showing that $\Ext^j_{(\fg,\fa_\0)} (P(\lambda), M) = 0$ for $j >0$, or equivalently,
\eq\label{eq:EOE}
\Ext^j_{(\fp^-,\fa_\0)} (P(\lambda),  I_\fa(\lambda_{0})) = 0,
\endeq
thanks to Frobenius reciprocity. 

From Proposition~\ref{P:projprop}, $P(\lambda)$ is a direct summand of the projective module $N := U(\fg)\otimes_{U(\fg_\0)} P_\0(\sigma)$ for some $\sigma\in \Lambda_0$.
Consider the following spectral sequence from Theorem \ref{thm:LHSSS}:
\eq\label{eq:EOSS3}
E_2^{i,j} = \Ext^i_{(\fa, \fa_\0)}(\CC, \Ext^j_{U(\fn^-)}(N,\CC) \otimes I_\fa(\mu_{0})) \Rightarrow \Ext^{i+j}_{(\fp^-, \fa_\0)}(N, I_\fa(\mu_{0})). 
\endeq
Note that $E_2^{i,j} = 0$ for $i>0$ due to injectivity of $I_\fa(\mu_{0})$.
As long as we can prove that $E_2^{i,j} = 0$ for $j>0$ as well, then $\Ext^n_{(\fp^-,\fa_\0)}(N, I_\fa(\mu_{0})) =0$ for $n
>0$, and thus \eqref{eq:EOE} follows.

Observe that $P_\0(\sigma)$ has a filtration with sections of the form $U(\fg_\0)\otimes_{U(\fp_\0)} L_\0 (\delta)$ for some $\delta\in \Lambda_0$. Then, $N$ has a filtration with sections of the form
\eq
U(\fg)\otimes_{U(\fg_\0)} U(\fg_\0)\otimes_{U(\fp_\0)} L_\0 (\delta)
\cong U(\fg)\otimes_{U(\fp_\0)} L_\0 (\delta)
\cong U(\fg)\otimes_{U(\fp)} (U(\fp)\otimes_{U(\fp_\0)} L_\0 (\delta)).
\endeq
Since such $U(\fg)\otimes_{U(\fp)} (U(\fp)\otimes_{U(\fp_\0)} L_\0 (\delta))$ is free over $U(\fn^-)$, it follows that 
$\Ext^j_{U(\fn^-)}(N, \CC) = 0$ for $j>0$. The statement of the theorem is now proved.
\endproof

\subsection{} Next we prove a variant of a lemma for $\ninj(\lambda)$ that is inspired by \cite[II Proposition 2.14]{Ja03}.
\begin{lem}\label{lem:Jantzen}
Let $\fp$ be a principal parabolic subalgebra of $\fg$, and let $V$ be a module in ${\cO^{\fp}}$ which satisfies that
$\Hom_{\cO^\fp}(L(\mu), V) \neq 0$ for some $\mu \in \Lambda$.
Suppose that
\begin{enumerate}
\item[(a)] $\Hom_{\cO^\fp}(L(\nu), V) = 0$ for all $\nu < \mu$,
\item[(b)] $\Ext^1_{\cO^\fp}(\Dirr(\nu), V) = 0$ for all $\nu \leq \mu$,
\end{enumerate}
then $\ninj(\mu) \hookrightarrow V$.
\end{lem}
\proof
Since $\Hom_{\cO^\fp}(L(\mu), V) \neq 0$, there is a short exact sequence $0 \to L(\mu) \to \ninj(\mu) \to Q \to 0$ for some $Q \in \cO^\fp$. From the long exact sequence 
in cohomology, one gets an exact sequence
\eq
0\to \Hom_{\cO^\fp}(Q,V) \to \Hom_{\cO^\fp}(\ninj(\mu),V) \to \Hom_{\cO^\fp}(L(\mu),V) \to \Ext^{1}_{\cO^\fp}(Q,V).
\endeq 
We are done as long as $\Ext^1_{\cO^\fp}(Q,V) = 0$, which would follow if 
\eq
\Ext^1_{\cO^\fp}(L(\sigma),V) = 0\quad\textup{for any composition factor } L(\sigma) \textup{ of }Q. 
\endeq
Note that $\sigma\leq \mu$. For each $\sigma$,  consider the short exact sequence $0 \to N \to \Dirr(\sigma) \to L(\sigma) \to 0$.  Applying the left exact functor $\Hom_{\cO^\fp}(-,V)$ and using 
the fact that $\Hom_{\cO^\fp}(N,V) = 0$ since any composition factor of $N$ is of the form $L(\gamma)$ for some $\gamma < \sigma \leq \mu$, one has an exact sequence
\eq
0\to  \Ext^1_{\cO^\fp}(L(\sigma),V) \to \Ext^1_{\cO^\fp}(\Dirr(\sigma),V).
\endeq
Furthermore, $\Ext^1_{\cO^\fp}(\Dirr(\sigma),V) = 0$ by (b) since  $\sigma  \leq \mu$. This concludes the proof.
\endproof

One can now extend the Ext-orthogonality result to a criteria for modules in $\cO^\fp$ to have a $\ninj$-filtration.

\thm \label{T:Jantzenlem}
Let $V\in \cO^\fp$. If $\dim \Hom_{\cO^\fp}(\ninj(\gamma), V) < \infty$ for all $\gamma \in \Lambda$. Then the following are equivalent:
\begin{enumerate}
\item[(a)] $V$ has a $\ninj$-filtration.
\item[(b)] $\Ext^n_{\cO^\fp}(\Dirr(\gamma), V) = 0$ for all $\gamma \in \Lambda$ and $n > 0$.
\item[(c)] $\Ext^1_{\cO^\fp}(\Dirr(\gamma), V) = 0$ for all $\gamma \in \Lambda$.
\end{enumerate}
Furthermore, if $I(\lambda)$ is the injective hull of $L(\lambda)$ in ${\mathcal O}^{\mathfrak p}$, then one has the following reciprocity law: 
$$[I(\lambda): \ninj(\gamma)]=[\Dirr(\gamma):L(\lambda)]$$ 
for $\lambda, \gamma\in \Lambda$. 
\endthm
\proof
The implication  $(a) \Rightarrow (b)$ is a consequence of Proposition~\ref{prop:exto}.
The implication  $(b) \Rightarrow (c)$ is clear. It remains to show that $(c) \Rightarrow (a)$.

First, consider those weights $\mu \in \Lambda$ such that
\eq\label{eq:minmu}
\Hom_{\cO^\fp}(L(\mu), V) \neq 0,
\quad
\Hom_{\cO^\fp}(L(\nu), V) = 0\textup{ for all }{\nu} < {\mu}.
\endeq
Since $V$ has finitely many composition factors by Theorem~\ref{T:finitecompfactors}, there is a minimal weight $\mu$ satisfying \eqref{eq:minmu} in the proset. 
That is, there is no $\sigma \in \Lambda$ such that $\={\sigma} < \={\mu}$, $\Hom_{\cO^\fp}(L(\sigma), V) \neq 0$, and $\Hom_{\cO^\fp}(L(\nu), V) = 0$ for all $\nu < \sigma$.

For any $\nu \leq \mu$, by (c) we have $\Ext^1_{\cO^\fp}(\Dirr(\nu), V) = 0$, and thus $\ninj(\mu) \hookrightarrow V$ by Lemma \ref{lem:Jantzen}, i.e., $V$ has a submodule $V'$ that is isomorphic to $\ninj(\mu)$.
By Proposition \ref{prop:exto}, $V'$ satisfies (b). Also, $V'' := V/V'$ satisfies (c). 
Note that $\Hom_{\cO^\fp}(\Dirr(\nu), V'') \cong \Hom_{\cO^\fp}(\Dirr(\nu), V)$ for all $\nu \neq \mu$ whereas
$\dim \Hom_{\cO^\fp}(\Dirr(\mu), V'') =\dim \Hom_{\cO^\fp}(\Dirr(\mu), V)-1$.
We can then iterate this construction to obtain a $\ninj$-filtration for $V$.

Finally, observe that by (c), $I(\lambda)$ has a $\ninj$-filtration, and one can count the number of times such a factor occurs by apply $\text{Hom}_{\cO^\fp}(\Dirr(\gamma),-)$ to short exact sequences of modules 
with a $\ninj$-filtration. Therefore, 
\[
[I(\lambda): \ninj(\gamma)]=\dim \text{Hom}_{\cO^\fp}(\Dirr(\gamma),I(\lambda))=[\Dirr(\gamma):L(\lambda)]. 
\]
\endproof

 \subsection{Proof of Theorem~\ref{T:stratification}}
\proof
We check the conditions given in Definition~\ref{def:ssc}. To show that 
$\cO^\fp$ is locally Artinian, it suffices to show that any $M\in \cO^\fp$ is a union of submodules of finite length. This follows since 
$M$ has finitely many composition factors by Theorem~\ref{T:finitecompfactors}. 

 For Conditions (a)--(b), 
the set 
$\Lambda$
has a structure of interval-finite proset given in \eqref{eq:Ofaposet}. 
 For Condition (c), we set  $\nabla(\lambda) = \ninj(\lambda)$ as in \eqref{eq:ninjDirr}.
For any $\mu = (\mu_0,j)$ such that  $\Pi^jL(\mu_0)$ is a composition factor of $\Pi^{j}\cF(\Hom_{U(\fp^-)} (U(\fg), I_\fa(\lambda_0)))$,
$\mu_0$ must be of the form $\mu_0 = \sigma - \gamma$ for some weight $\sigma$ of $I_\fa(\lambda_0)$ and some weight $\gamma$ of $U(\fn^-)$. 
Moreover, 
since any weight in $I_\fa(\lambda_0)$ has the same $\cH$-value (i.e., every weight of $I_\fa(\lambda_0)$ has $\cH$ value $\cH(\lambda_{0}))$,
\eq
\overline{\mu}_0 = \cH(\sigma -\gamma) = \cH(\sigma) -\cH(\gamma)  = \overline{\lambda_0} -\cH(\gamma)  \leq \overline{\lambda_0}.
\endeq
Hence, $\mu_0 \leq \lambda_0$.

 For Condition (d), we will construct a $\nabla$-filtration of $I(\lambda)$ or $\Pi I(\lambda$) using Theorem~\ref{T:Jantzenlem}. Since $I(\lambda)$ is injective 
 in ${\mathcal O}^{\mathfrak p}$, it follows that $\Ext^n_{\cO^\fp}(\Dirr(\gamma), I(\lambda)) = 0$ for all $\gamma \in \Lambda$ and $n > 0$. Therefore, 
 $I(\lambda)$ has a $\ninj$-filtration. The proof of Theorem~\ref{T:Jantzenlem} demonstrates that the filtration can be constructed with the necessary properties. 
\endproof

\section{Bound on the Complexity}

\subsection{Complexity} The {\em complexity} of a module was first defined by Alperin for finite groups (cf. \cite{Al77}). This notion was later extended to finite group schemes and with the deep result of 
Friedlander and Suslin ~\cite{FS97}, one can identify the complexity with the dimension of the support variety of a module. For classical Lie superalgebras, this concept was used in work by Boe, Kujawa, and Nakano~\cite{BKN10b} to study the category of finite-dimensional representations. Coulembier and Serganova extended this definition for the Category ${\mathcal O}$ for ${\mathfrak g}=\mathfrak{gl}(m|n)$~\cite{CS17}. 
We use this definition for ${\mathcal O}^{\mathfrak p}$, where ${\mathfrak p}$ is a principal parabolic subalgebra of ${\mathfrak g}$.

In order to define the complexity of a module, recall that given a sequence $\{V_i\}_{i\geq 0}$ of finite-dimensional complex vector spaces, its vector space rate of growth $r(V_i)$ is the smallest positive integer $c$ such that $\dim V_i\leq Ci^{c-1}$ for some positive constant $C$.    

\begin{defn} Let ${\mathfrak p}$ be a principal parabolic subalgebra and $M$ be a module in ${\mathcal O}^{\mathfrak p}$. The {\em complexity} of $M$ is defined as 
$$c_\fp(M) = r(\Ext^i_{\cO^\fp}(M, \oplus_{\lambda\in \Lambda}  L(\lambda))).$$ 
\end{defn} 
According to Theorem~\ref{T:cohomology}, the complexity can be expressed in terms of relative cohomology 
\begin{equation} 
c_\fp(M) = r(\Ext^i_{(\fg, \fl_\0)}(M,  \oplus_{\lambda\in \Lambda}  L(\lambda))).
\end{equation} 

\subsection{Universally $\0$-bounded} Let ${\mathfrak p}$ be a principal parabolic in ${\mathfrak g}$. From Lemma~\ref{L:ptop0}, if $L(\lambda)$ is a simple module in ${\mathcal O}^{\fp}$ then this module is in ${\mathcal O}^{\fp_{\0}}$.  Therefore, in the Grothendieck group $[L(\lambda)]=\sum_{\mu\in \Lambda_{\0}}n_{\lambda,\mu} [L_{\0}(\mu)]$. We say that the the simple modules in ${\mathcal O}^{\fp}$ are {\em universally $\0$-bounded} if and only if there exists a $C>0$ such that $n_{\lambda,\mu}\leq C$ for all $\lambda\in \Lambda$ and $\mu\in \Lambda_{\0}$. 

As we will see in the next result, in order to show that the simple module in ${\mathcal O}^{\fp}$ are {\em universally $\0$-bounded}, one needs to use important results involving parabolic Category ${\mathcal O}$ for semisimple Lie algebras. 

\begin{prop} Let ${\mathfrak p}$ be a principal parabolic in ${\mathfrak g}$. Then the simple modules in ${\mathcal O}^{\fp}$ are universally $\0$-bounded. 
\end{prop} 

\begin{proof} Let $L(\lambda)$ be a simple module in ${\mathcal O}^{\mathfrak p}$. Then $L(\lambda)$ is in ${\mathcal O}^{{\mathfrak p}_{\0}}$, and there exists a simple module in 
${\mathcal O}^{{\mathfrak p}_{\0}}$, $L_{\0}(\sigma)$, where $\sigma\in \Lambda_{\0}$ such that there is a monomorphism $L_{\0}(\sigma)\hookrightarrow L(\lambda)$. Therefore, one has 
\begin{equation} 
0\neq \text{Hom}_{U({\mathfrak g}_{\0})}(L_{\0}(\sigma),L(\lambda))\cong \text{Hom}_{U({\mathfrak g})}(U({\mathfrak g})\otimes_{U({\mathfrak g}_{\0})} L_{\0}(\sigma),L(\lambda)).
\end{equation} 
It follows that there exists a surjective $U({\mathfrak g})$-homomorphism $U({\mathfrak g})\otimes_{U({\mathfrak g}_{\0})} L_{\0}(\sigma)\twoheadrightarrow L(\lambda)$. 
This restricts to a surjective $U({\mathfrak g}_{\0})$-homomorphism. 
Let $Z_{\0}(\sigma)$ be a parabolic Verma module in ${\mathcal O}^{{\mathfrak p}_{\0}}$. 
Then there exists a surjective homomorphism $Z_{\0}(\sigma)\twoheadrightarrow L_{\0}(\sigma)$, 
which means there exists a surjective homomorphism $U({\mathfrak g})\otimes_{U({\mathfrak g}_{\0})} Z_{\0}(\sigma)\twoheadrightarrow L(\lambda)$. 
As a $U(\fg_{\0})$-module, one 
has 
\begin{equation}\label{E:moduleupper}
U({\mathfrak g})\otimes_{U({\mathfrak g}_{\0})} Z_{\0}(\sigma)\cong \Lambda^{\bullet}({\mathfrak g}_{\1})\otimes_{{\mathbb C}}Z_{\0}(\sigma)
\cong U({\mathfrak g}_{\0})\otimes_{U({\mathfrak p}_{\0})} [\Lambda^{\bullet}({\mathfrak g}_{\1})\otimes L_{{\mathfrak p}_{\0}}(\sigma)]. 
\end{equation} 
The last isomorphism uses the tensor identity. 

The module $\Lambda^{\bullet}({\mathfrak g}_{\1})\otimes L_{{\mathfrak p}_{\0}}(\sigma)$ is a finite-dimensional module over $U({\mathfrak l}_{\0})$. One can show that the composition multiplicities are 
bounded by the number of weights of $\gamma=\sigma+\mu$ where $\mu$ is a weight of $\Lambda^{\bullet}({\mathfrak g}_{\1})$, and $\gamma$ is dominant with respect to ${\mathfrak l}_{\0}$. This number is clearly 
bounded by the dimension of $\Lambda^{\bullet}({\mathfrak g}_{\1})$. Now for each $\gamma$, the composition multiplicities of [$Z_{\0}(\gamma):L_{\0}(\tau)]$ are bounded by Proposition~\ref{Cat O-Liealg prop}(c).
Hence, the composition factor multiplicities of the $U({\mathfrak g}_{\0})$-module in (\ref{E:moduleupper}) is universally bounded for all 
$\sigma$, and thus the simple modules in ${\mathcal O}^{\fp}$ are universally $\0$-bounded. 
\end{proof} 


\subsection{Proof of Theorem~\ref{T:complexity}}

We can now show that the complexity of all modules in $\cO^\fp$ where ${\mathfrak p}$ is a principal parabolic are  bounded by $\dim {\mathfrak g}_{\1}$. 

\thm
If $\fp$ is a principal parabolic subalgebra of $\fg$ then $c_\fp(M) \leq \dim \fg_\1$ for all $M \in \cO^\fp$.
\endthm
\proof Set $T= \Hom_\CC(M, \bigoplus_{\lambda\in \Lambda}L(\lambda))$. By using Proposition~\ref{P:ExtC} and Theorem~\ref{thm:Maurer}
\eq
\begin{split}
c_\fp(M) &= r(\Ext^{\bullet}_{(\fg, \fl_\0)}(M,  \bigoplus\nolimits_{\lambda\in \Lambda}  L(\lambda)) 
\\
&= r(\Ext^{\bullet}_{(\fg, \fl_\0)}(\CC,  T))
\\
&\leq  r\big(
\bigoplus\nolimits_{p+q=\bullet} \opH^p(\fg, \fg_\0, T) \otimes \opH^q(\fg_\0, \fl_\0, \CC )
\big) 
\\
&= r\big(
\Ext^{\bullet}_{(\fg, \fg_\0)}(M,  \bigoplus\nolimits_{\lambda\in \Lambda}  L(\lambda))
\big) .
\end{split}
\endeq
There exists a projective resolution for ${\mathbb C}$ given by 
$$\dots\rightarrow D_{2}\rightarrow D_{1}\rightarrow D_{0}\rightarrow {\mathbb C} \rightarrow 0$$ 
where $D_{j}=U({\mathfrak g})\otimes_{U({\mathfrak g}_{\0})} S^{j}({\mathfrak g}_{\1}^{*})$ where $j\geq 0$ (cf. \cite[Proposition 2.4.1]{BKN10b}). 
One can tensor this resolution by $M$ to get a relative projective resolution for $(U({\mathfrak g}),U({\mathfrak g}_{\0}))$: 
$$\dots\rightarrow D_{2}\otimes M\rightarrow D_{1}\otimes M\rightarrow D_{0}\otimes M\rightarrow M \rightarrow 0.$$ 
Now apply the functor $\text{Hom}_{({\mathfrak g},{\mathfrak g}_{\0})}(-,\oplus_{\lambda\in \Lambda}  L(\lambda))$ to this resolution to get a complex with cochains: 
\eq
C_{j}=\text{Hom}_{({\mathfrak g},{\mathfrak g}_{\0})} \Big(
D_{j}\otimes M,\bigoplus_{\lambda\in \Lambda}  L(\lambda)
\Big)
\cong 
\text{Hom}_{({\mathfrak g}_{\0},{\mathfrak g}_{\0})}\Big(
S^{j}({\mathfrak g}_{\1}^{*})\otimes M, \bigoplus_{\lambda\in \Lambda}L(\lambda)
\Big).
\endeq
Since $\Ext^{j}_{(\fg, \fg_\0)}(M,  \bigoplus_{\lambda\in \Lambda}  L(\lambda))$ is a subquotient of $C_j$ for $j\geq 0$, one has 
\begin{equation} 
 r\Big(
 \Ext^{\bullet}_{(\fg, \fg_\0)}(M,  \bigoplus_{\lambda\in \Lambda}  L(\lambda))
 \Big)
 \leq 
 r(C^{\bullet})
 \leq 
r\Big( 
\text{Hom}_{({\mathfrak g}_{\0},{\mathfrak g}_{\0})}\big(
S^{\bullet}({\mathfrak g}_{\1}^{*})\otimes M, \bigoplus_{\lambda\in \Lambda_{\0}}L_{\0}(\lambda)
\big)
\Big).
 \end{equation}  
The last inequality holds because the simple modules in ${\mathcal O}^{\fp}$ are universally $\0$-bounded. 

In the Grothendieck group $[M]=\sum_{\sigma\in \Gamma}n_{\sigma}[L_{\0}(\sigma)]$. Since $\Gamma$ is finite we can bound $n_{\sigma}$ where 
$\sigma\in \Gamma$. Therefore, 
\eq
r\Big( 
\text{Hom}_{({\mathfrak g}_{\0},{\mathfrak g}_{\0})}\big(
S^{\bullet}({\mathfrak g}_{\1}^{*})\otimes M, \bigoplus_{\lambda\in \Lambda_{\0}}L_{\0}(\lambda)
\big)
\Big) 
\leq 
r\Big( 
\text{Hom}_{({\mathfrak g}_{\0},{\mathfrak g}_{\0})}\big(
\bigoplus_{\sigma\in \Gamma} S^{\bullet}({\mathfrak g}_{\1}^{*})\otimes L_{\0}(\sigma), \bigoplus_{\lambda\in \Lambda_{\0}} L_{\0}(\lambda)
\big)
\Big). 
\endeq

For each $j\geq 0$, one has that $\oplus_{\sigma\in \Gamma} S^{j}({\mathfrak g}_{\1}^{*})\otimes Z_{\0}(\sigma)$ surjects onto  
$\oplus_{\sigma\in \Gamma} S^{j}({\mathfrak g}_{\1}^{*})\otimes L_{\0}(\sigma)$. Moreover, in the Grothendieck group
\eq
\Big[
\bigoplus_{\sigma\in \Gamma} S^{j}({\mathfrak g}_{\1}^{*})\otimes Z_{\0}(\sigma)
\Big]
\cong 
\Big[
\bigoplus_{\sigma\in \Gamma} U({\mathfrak g}_{\0})\otimes_{U({\mathfrak p}_{\0})} S^{j}({\mathfrak g}_{\1}^{*})\otimes \sigma 
\Big]
=\sum_{\gamma\in \Gamma^{\prime}} k_{\gamma}[L_{\0}(\gamma)]
\endeq
where $\Gamma^{\prime}$ is a finite subset of $\Lambda_{\0}$. 

The module $\bigoplus_{\sigma\in \Gamma} U({\mathfrak g}_{\0})\otimes_{U({\mathfrak p}_{\0})} S^{j}({\mathfrak g}_{\1}^{*})\otimes \sigma$ has a filtration with sections being parabolic Verma modules. 
The number of sections is equal to $|\Gamma| \cdot \dim S^{j}({\mathfrak g}_{\1}^{*})$. According to Proposition~\ref{Cat O-Liealg prop}(c), the number of composition factors in any parabolic Verma module is universally bounded by $D$ (which depends on ${\mathfrak g}_{\0}$ and the parabolic ${\mathfrak p}_{\0}$). This implies that $k_\gamma\leq D\cdot |\Gamma| \cdot \dim S^{j}({\mathfrak g}_{\1}^{*})$ for all 
$\gamma\in \Gamma^{\prime}$. 

Set $Q^{j}:=\bigoplus_{\lambda\in \Gamma} S^{j}({\mathfrak g}_{\1}^{*})\otimes Z_{\0}(\lambda)$. Then 
\eq
\begin{split}
\dim \text{Hom}_{({\mathfrak g}_{\0},{\mathfrak g}_{\0})} \Big(
Q^{j}, \bigoplus_{\lambda\in \Lambda_{\0}} L_{\0}(\lambda) 
\Big)
&\leq
D\cdot |\Gamma|\cdot \dim S^{j}({\mathfrak g}_{\1}^{*})\cdot 
\dim \text{Hom}_{({\mathfrak g}_{\0},{\mathfrak g}_{\0})} \Big(
\bigoplus_{\gamma\in \Gamma^{\prime}} L_{\0}(\gamma), \bigoplus_{\lambda\in \Lambda_{\0}} L_{\0}(\lambda)
\Big) 
\\
&= D\cdot |\Gamma|\cdot \dim S^{j}({\mathfrak g}_{\1}^{*})\cdot  |\Gamma^{\prime}|.
\end{split} 
\endeq
We can now put this information together to conclude that 
\eq
c_\fp(M)\leq r(S^{\bullet}({\mathfrak g}_{\1}^{*}) )=\dim {\mathfrak g}_{\1}. \qedhere
\endeq
\endproof
\appendix
\section{Almost Simple Lie Superalgebras} \label{S:QS} 
In the appendix, we compare  common setups in the study of finite-dimensional Lie superalgebras over $\CC$.
\subsection{Quasi-reductive Lie Superalgebras}
The complete list (see Table~\ref{T:simples} below) of simple Lie superalgebras consists of the Cartan series and the classical simple Lie superalgebras.
A finite-dimensional (not necessarily simple) Lie superalgebra $\fg$ is {\em quasi-reductive} (also called {\em classical}) if $\fg_\0$ is reductive and $\fg_\1$ is semisimple as a $\fg_\0$-module.
A classical simple Lie superalgebra $\fg$ is \emph{basic} if it has a nondegenerate invariant supersymmetric even bilinear form; while it belongs to the strange series if it is not basic.
A basic classical Lie superalgebra $\fg$ is said to be of {\em Type I} if  $\fg$ is ${\mathbb Z}$-graded with $\fg_\0 = \fg_0$ and $\fg_\1 = \fg_1 \oplus \fg_{-1}$; otherwise, $\fg$ is said to be of {\em Type II}.

\begin{table}[htp]
\caption{Classification of simple Lie superalgebras}
\label{T:simples}
\begin{center}

\begin{tabular}{|c|c|c|c|}
\hline
\multicolumn{4}{|c|}{Simple Lie superalgebras}                                                                                                       
\\ \hline
	\multicolumn{3}{|c|}{Classical simple Lie superalgebras}                                                                   
	& \multirow{4}{*}{\begin{tabular}{c}\\Cartan series\\ \\
		$\begin{array}{l}
		W(n)
		\\
		S(n), \overline{S}(n)
		\\
		H(n)
		\end{array}$
	\end{tabular}} 
\\ \cline{1-3}
\multicolumn{1}{|c|}{}        
& \multicolumn{1}{c|}{Basic classical Lie superalgebras} & \multicolumn{1}{c|}{Strange series} &                                
\\ \cline{1-3}
	\multicolumn{1}{|c|}{Type I}  
	& \multicolumn{1}{c|}{
	$\begin{array}{l}
		A(m|n) =\begin{cases}
		\fsl(m|n) &\tif m\neq n;
		\\
		\mathfrak{psl}(n|n) &\tif m=n.
		\end{cases}
		\\
		C(n) = \fosp(2|2n-2), n\geq 2
	\end{array}$
	}                        
	& \multicolumn{1}{c|}{         
	 $\begin{array}{l}
 		\fp(n)
 	\end{array}$     
 	}
	&                                
	\\ \cline{1-3}
	\multicolumn{1}{|c|}{Type II} 
	& \multicolumn{1}{c|}{
	$\begin{array}{l}
		B(m|n) = \fosp(2m+1|2n)
		\\
		D(m|n) = \fosp(2m|2n), m\geq 2
		\\
		D(2|1;\alpha), F(4), G(3)
	\end{array}$
	}                       
	& \multicolumn{1}{c|}{
	 $\begin{array}{l}
 		\fpsq(n)
 	\end{array}$	
	}               
	&                   
	\\
	\hline            
\end{tabular}
\end{center}
\end{table}
Here, $\mathfrak{q}(n)$ denotes the (non-simple) Lie superalgebra with even and odd parts $\mathfrak{gl}_n(\CC)$; while $\mathfrak{psq}(n)$ is its simple subquotient. 
We denote the family of Type P Lie superalgebras by ${\mathfrak p}(n)$ and its enlargement $\widetilde{{\mathfrak p}}(n)$, with even parts $\fsl_n(\CC)$ and $\fgl_n(\CC)$, respectively.

\subsection{Almost Simple Lie Superalgebras}\label{sec:almostsimple}

It is well-known  that for basic Lie superalgebras and for $\fq(n)$, there is a structure theory similar to the one for simple Lie algebras (see \cite{CW12}).
In this paper we will focus on a larger class of Lie superalgebras, which consists of all Lie superalgebras that afford Bott-Borel-Weil (BBW) parabolic subalgebras 
constructed in \cite{GGNW21}. 
We will call them {\em almost simple} Lie superalgebras from now on. 
Almost simple Lie superalgebras (cf. \cite[Section 7]{GGNW21}) are quasi-reductive, and they include all classical simple Lie superalgebras, as well as certain non-simple Lie superalgebras (see the Venn diagram in Figure~\ref{T:QSA} below) that are close enough to being simple.

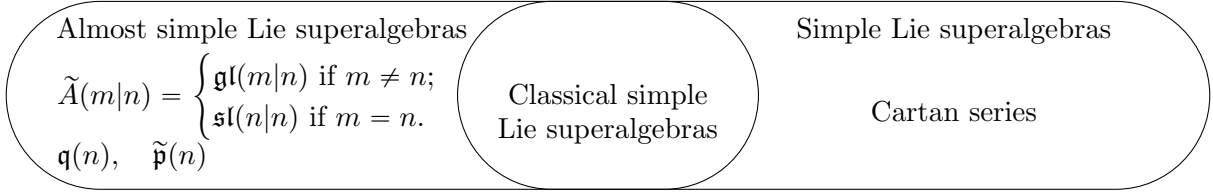
\begin{figure}[htp]
  \caption{Almost simple Lie superalgebras versus simple Lie algebras}
  \label{T:QSA}
  \centering
\begin{tikzpicture}
\tikzset{ outline/.style={thick}}
\draw [rounded corners=1.25cm] (-8,0.5) rectangle ++(10,2.5)  ;
\draw [rounded corners=1.25cm] (-2,0.5) rectangle ++(10,2.5)  ;
	\draw[outline] (-4.6, 2.6) node {\begin{tabular}{c}Almost simple Lie superalgebras\end{tabular}};
	\draw[outline] (4.6,2.6) node {\begin{tabular}{c}Simple Lie superalgebras\end{tabular}};
	\draw[outline] (0,1.5) node {\begin{tabular}{c}Classical simple\\Lie superalgebras\end{tabular}};
	\draw[outline] (4.6,1.5) node {\begin{tabular}{c}Cartan series\end{tabular}};
	\draw[outline] (-4.6,1.5) node {	
	$\begin{array}{l}
\widetilde{A}(m|n) 
=\begin{cases}
\fgl(m|n) ~\tif m\neq n;
\\
\fsl(n|n) ~\tif m=n.
\end{cases}
\\
\fq(n), \quad 
\widetilde{\fp}(n)
\end{array}$
};
\end{tikzpicture}
 \end{figure}
 
 \subsection{Calculations used for the finite generation of cohomology} 
 
 In this subsection, we present the calculations needed for the proof of Theorem \ref{T:finite generation} (see Section \ref{sec:proofB}). This is presented in a case by case manner. 
 
 \subsubsection{Type A}\label{SS:Type A} This includes ${\mathfrak g}=\mathfrak{gl}(m|n)$ and its variants. 
 For the sake of simplicity we will present the argument  for $\mathfrak{gl}(n|n)$. 
 For $\mathfrak{gl}(n|n)$, the subgroup $H$ of 
 $G_{\0}$ is the diagonal copy of the maximal torus of $GL_n$. 
 Let $\widehat{T}$ to be a one-dimensional torus in $H$ as 
 \eq
 \widehat{T}=\langle t \rangle,
 \quad
 \textup{with}
 \quad
t=  (\text{diag}\{a^{r},a^{r-2},\dots, a^{-r}\}, \text{diag}\{a^{r},a^{r-2},\dots, a^{-r}\}),
 \endeq
for $a \in \CC^*$. Here, $r=\lceil \frac{n+1}{2} \rceil$.

 The simple roots for $G_{\0}\cong GL_{n}\times GL_{n}$ are two copies of $\{\alpha_{1},\alpha_{2},\dots,\alpha_{n-1}\}$ where 
 $\alpha_{j}=\epsilon_{j}-\epsilon_{j+1}$ for $j=1,2,\dots,n-1$. Then $\alpha_{j}(t)=2$ for all $j$, and $\alpha(t)>0$ for all $\alpha\in 
 \Phi_{\0}^{+}$.  Set ${\mathfrak g}_{\0}\cong {\mathfrak n}^{-}_{\0}\oplus {\mathfrak t} \oplus {\mathfrak n}_{\0}^{+}$ be the triangular decomposition of 
 ${\mathfrak g}_{\0}$ relative to the root system $\Phi_{\0}$. Therefore, for any given weight $\lambda$, 
\eq
\text{Hom}_{\widehat{T}}({\mathbb C},\lambda \otimes U({\mathfrak n}_{\0}^{-}))\cong \text{Hom}_{\widehat{T}}(-\lambda,U({\mathfrak n}_{\0}^{-}))
 \cong \text{Hom}_{\widehat{T}}(\lambda,U({\mathfrak n}_{\0}^{+})).
 \endeq
 As a $\widehat{T}$-module, $U({\mathfrak n}_{\0}^{+})\cong S^{\bullet}({\mathfrak n}_{\0}^{+})$. 
 It follows that because the roots are positively graded 
 there are only finitely many weights in $S^{\bullet}({\mathfrak n}_{\0}^{+})$ that have weight $\lambda(t)$, thus  
 $\text{Hom}_{\widehat{T}}({\mathbb C},\lambda \otimes U({\mathfrak n}_{\0}^{-}))$ is finite-dimensional. This finishes the proof of Theorem \ref{T:finite generation}, for 
 $\mathfrak{gl}(n|n)$. 
 
 For the general case for $\mathfrak{gl}(m|n)$, one can embed a torus isomorphic to $\widehat{T}$ and enlarge this to a torus containing the remaining part of the 
 maximal torus for $G_{\0}$, in order to induce a positive grading on the roots. 
 Note that the element $t$ has determinant one. 
 This will allow one to prove the result for $\mathfrak{sl}(m|n)$ and its quotients. 
 
 \subsubsection{Type Q}  
 For ${\mathfrak q}(n)$ and its variants, the same argument used for $\mathfrak{gl}(n|n)$ can be applied. In this case $G_{\0}\cong GL_{n}$ and $H$ is the maximal torus of $G_{\0}$. 
 
\subsubsection{Type $P$} 
The methods that we employed above can be applied for Lie superalgebras of Type $P$. 
Let ${\mathfrak g}=\widetilde{\mathfrak p}(n)$ be the Lie superalgebra of type $P$ with 
${\mathfrak g}_{\0}\cong \mathfrak{gl}(n)$, and ${\mathfrak g}_{\1}\cong S^{2}(V)\oplus \Lambda^{2}(V^{*})$. 
One can modify this argument to handle ${\mathfrak p}(n)$ where the zero component is 
isomorphic to $\mathfrak{sl}(n)$. 
The relevant details about Type P can be found in \cite[Section 8]{BKN10a}. 

For the sake of simplicity, assume that $n$ is even. 
The $n$ odd case follows by similar reasoning and the verification is left to the reader. 
The subgroup $H$ of $G_{\0}$ has dimension $\frac{n}{2}$. 
There exists a one-dimensional subgroup $\widehat{T}$ where $\widehat{T}=\langle t \rangle$ with 
\eq
t=  (\text{diag}\{a^{\frac{n}{2}},a^{\frac{n}{2}-1},\dots, a^{1},a^{-1},\dots, a^{-\frac{n}{2}}\}, \text{diag}\{-a^{\frac{n}{2}},-a^{\frac{n}{2}-1},\dots, -a^{1},-a^{-1},\dots, -a^{-\frac{n}{2}}\}),  
 \endeq
for $a \in \CC^*$.

The simple roots for $G_{\0}\cong GL_{n}$ are $\{\alpha_{1},\alpha_{2},\dots,\alpha_{n-1}\}$. One can see that 
$\alpha(t)>0$ for all $\alpha\in \Phi_{\0}^{+}$, and $\text{Hom}_{\widehat{T}}({\mathbb C},\lambda \otimes U({\mathfrak n}_{\0}^{-}))$ is finite-dimensional for a fixed weight $\lambda$. 
 
 \subsubsection{Type BC} 
 In this situation, we can first assume that ${\mathfrak g}=\mathfrak{osp}(2n+1|2n)$. 
 The maximal torus for $G_{\0}\cong O_{2n+1}\times Sp_{2n}$ has the form 
\eq
T=\{ (\text{diag}\{a_{1}^{r_{1}},\dots, a_{n}^{r_{n}}, 1, a_{1}^{-r_{1}},a_{2}^{-r_{2}},\dots,a_{n}^{-r_{n}} \}, \text{diag}\{b_{1}^{t_{1}},b_{2}^{t_{2}},\dots, b_{n}^{t_{n}},b_{1}^{-t_{1}},\dots, b_{n}^{-t_{n}}\}) \}
\endeq
 where $a_{j},b_{j}\in {\mathbb C}^{*}$, and $r_{j}, t_{j} \in {\mathbb Z}$ for $j=1,2,\dots,n$.  
 Set  $\widehat{T}=\langle t \rangle$, where
\eq
t = (\text{diag}\{a^{n},a^{n-1},\dots, a^{1},1, a^{-n},a^{-(n-1)},\dots,a^{-1} \}, \text{diag}\{a^{n},a^{n-1},\dots, a^{1},a^{-n},\dots, a^{-1}\}) 
\endeq
for $a \in \CC^*$.
From \cite[Sections 8.9 and 8.10]{BKN10a}, one can see that $\widehat{T}\leq H$. 
 
 The simple roots for $G_{\0}$ are given by 
\[
\{\epsilon_{1}-\epsilon_{2},\dots,\epsilon_{n-1}-\epsilon_{n},\epsilon_{n}\}\cup \{\delta_{1}-\delta_{2},\dots,\delta_{n-1}-\delta_{n},2\delta_{n} \}.
\]
 It follows that $(\epsilon_{j}-\epsilon_{j+1})(t)=1$ and $(\delta_{j}-\delta_{j+1})(t)=1$ for $j=1,2,\dots,n-1$. 
 Moreover,  $\epsilon_{n}(t)=1$ and 
 $2\delta_{n}(t)=2$, thus $\alpha(t)>0$ for all $\alpha\in \Phi_{\0}^{+}$. 
 
 One can now use the argument given in Section~\ref{SS:Type A} to finish the proof of Theorem \ref{T:finite generation} for $\mathfrak{osp}(2n+1|2n)$. 
 The general case for $\mathfrak{osp}(2m+1,2n)$ can be handled by 
 using the same argument as in the general $\mathfrak{gl}(m|n)$ situation where you enlarge the torus $\widehat{T}$ to a larger torus.  
 
 \subsubsection{Type D} We can reduce our analysis to ${\mathfrak g}=\mathfrak{osp}(2n|2n)$. The general case can be handled by extending the torus in $H$. The maximal torus 
 in $G_{\0}\cong O(2n)\times Sp_{2n}$ has the form 
\eq
T=\{ (\text{diag}\{a_{1}^{r_{1}},\dots, a_{n}^{r_{n}}, a_{1}^{-r_{1}},a_{2}^{-r_{2}},\dots,a_{n}^{-r_{n}} \}, \text{diag}\{b_{1}^{t_{1}},b_{2}^{t_{2}},\dots, b_{n}^{t_{n}},b_{1}^{-t_{1}},\dots, b_{n}^{-t_{n}}\}) ,
\endeq
 where $a_{j},b_{j}\in {\mathbb C}^{*}$, and $r_{j}, t_{j} \in {\mathbb Z}$ for $j=1,2,\dots,n$.  
 Set  $\widehat{T}=\langle t \rangle$, where
\eq
t= (\text{diag}\{a^{n},a^{n-1},\dots, a^{1},a^{-n},a^{-(n-1)},\dots,a^{-1} \}, \text{diag}\{a^{n},a^{n-1},\dots, a^{1},a^{-n},\dots, a^{-1}\})
\endeq
for $a\in \CC^*$.
Then, $\widehat{T}\leq H$. 
The simple roots for $G_{\0}$ are 
 $$\{\epsilon_{1}-\epsilon_{2},\dots,\epsilon_{n-1}-\epsilon_{n},\epsilon_{n-1}+\epsilon_{n},\}\cup \{\delta_{1}-\delta_{2},\dots,\delta_{n-1}-\delta_{n},2\delta_{n} \}.$$ 
 One can check that $\alpha(t)>0$ for all $\alpha\in \Phi_{\0}^{+}$ to finish the argument. 
  
 \subsubsection{$D(2,1,\alpha)$} 
 In this case,
 $G_{\0}\cong SL_{2}\times SL_{2} \times SL_{2}$. 
 The maximal torus is given by 
 \eq
 T=\{(a_{1}^{t_{1}},a_{2}^{t_{2}}, a_{3}^{t_{3}}) ~|~\ a_{j}\in {\mathbb C}^{*},\ t_{j}\in {\mathbb Z},\ \text{for $j=1,2,3$} \}.
 \endeq
 From \cite[Sections 8.9--8.10]{BKN10a}, one can choose the generic element to be of the form 
\eq
x_{0}=c(x_{\omega}\otimes y_{\omega} \otimes z_{-\omega})+d(x_{-\omega}\otimes y_{-\omega} \otimes z_{\omega})
\endeq 
where $\omega$ is a fundamental weight for $SL_{2}$. With this choice of $x_{0}$, one has 
\eq
\widehat{T}=\langle t \rangle,
\quad
\textup{where}
\quad
t=  (b_{1},b_{2},b_{1}b_{2})
\endeq 
for $b_1, b_2 \in \CC^*$.
One sees that $\widehat{T}$ is a two-dimensional torus in $H$. 
 
The simple roots for $G_{0}$ are three copies of the root system $A_{1}$. The $T$-weight spaces of $U({\mathfrak n}_{\0}^{+})$ are one-dimensional, and for a fixed $\lambda
=m_{1}\alpha_{1}+m_{2}\alpha_{2} +m_{3}\alpha_{3}$, it follows that $t$ on a weight vector $x_{\lambda}$ of weight $\lambda$ will yield 
\eq
t.x_{\lambda}=b_{1}^{2m_{1}+2m_{3}}b_{2}^{2m_{2}+2m_{3}}x_{\lambda}.
\endeq 
Set $\mu_{1}=2m_{1}+2m_{3}$ and $\mu_{2}=2m_{2}+2m_{3}$. 
For $n_{j}\geq 0$ for $j=1,2,3$, there are only finitely many possibilities for $\mu_{1}=2n_{1}+2n_{3}$ and $\mu_{2}=2n_{2}+2n_{3}$. 
Consequently, $\text{Hom}_{\widehat{T}}(\lambda,U({\mathfrak n}_{\0}^{+}))$ is finite-dimensional. 
 
 \subsubsection{$G(3)$ and $F(4)$} The arguments are similar for $G(3)$ and $F(4)$. We will sketch the $G(3)$ argument. 
 For $G(3)$, the maximal torus for $G_{0}=SL_{2}\times G_{2}$ is three-dimensional: 
\eq
T=\{(a^{r},b_{1}^{t_{1}}, b_{2}^{t_{2}}) ~|~\ a, b_{j}\in {\mathbb C}^{*},\ r, t_{j}\in {\mathbb Z},\ \text{for $j=1,2$} \}.
\endeq
We assume that the torus for $G_{2}$ is calibrated with respect to the simple roots $\Delta=\{\alpha_{1},\alpha_{2}\}$ for $G_{2}$ where $\alpha_{2}$ is long. 

In this case the generic point is $x_{0}=c(x_{-\omega}\otimes y_{\alpha_{2}})+d (x_{\omega}\otimes y_{-\alpha_{2}})$.  Set $\widehat{T}=T\cap H$. Then 
\eq
\widehat{T}=\langle t \rangle,
\quad
\textup{where}
\quad
t=  (b_{2},b_{1},b_{2})
\endeq 
for $b_1, b_2 \in \CC^*$.
One sees that $\widehat{T}$ is a two-dimensional torus in $H$. 

Now consider the $T$-weight spaces of $U({\mathfrak n}_{\0}^{+})$. For a fixed $\lambda
=m\alpha+m_{1}\alpha_{1} +m_{2}\alpha_{2}$, it follows that $t$ on a weight vector $x_{\lambda}$ of weight $\lambda$ will yield 
$t.x_{\lambda}=b_{1}^{m_{1}}b_{2}^{2m+m_{2}}x_{\lambda}$. It follows that $\text{Hom}_{\widehat{T}}(\lambda,U({\mathfrak n}_{\0}^{+}))$ is finite-dimensional by 
using the same reasoning as in the $D(2,1,\alpha)$-case.

\end{document}